\pdfoutput=1 
\documentclass[11pt,a4paper,oneside]{amsart} 

\usepackage{amsmath}
\usepackage{amsfonts}
\usepackage{amssymb}
\usepackage{amsthm}
\usepackage{amstext}
\usepackage{braket}
\usepackage{mathtools}
\usepackage{array}
\usepackage{enumitem}
\usepackage{mathrsfs}

\usepackage[english]{babel}
\usepackage[utf8]{inputenc}
\usepackage[sc]{mathpazo} %Font Palatino
\usepackage{setspace}
\usepackage{enumitem}
\usepackage{microtype}
\usepackage[usenames,dvipsnames,svgnames,table]{xcolor}
\numberwithin{equation}{section}

\usepackage{tikz}
\usepackage{textcomp}

\theoremstyle{plain}
\newtheorem{thm}{\textbf{Theorem}}[section]
\newtheorem{prop}[thm]{\textbf{Proposition}}
\newtheorem{cor}[thm]{\textbf{Corollary}}	
\newtheorem{lemma}[thm]{\textbf{Lemma}}

\theoremstyle{definition}

\theoremstyle{remark}
\newtheorem{rmk}[thm]{\textbf{Remark}}

\theoremstyle{plain}

\newcommand{\thistheoremname}{}
\newtheorem*{genericthm*}{\thistheoremname}
\newenvironment{namedthm*}[1]
  {\renewcommand{\thistheoremname}{#1}%
   \begin{genericthm*}}
  {\end{genericthm*}}

\newcommand{\df}{\overset{\text{def}}{=}}

\usepackage[alphabetic]{amsrefs}
\usepackage{hyperref}      

\usepackage[a4paper,top=3cm,bottom=3cm,left=0.5cm,right=0.5cm,bindingoffset=0mm]{geometry}
\title{Asymptotic syzygies and higher order embeddings}
\author{Daniele Agostini}
\address{
  Humboldt-Universit\"{a}t zu Berlin\\
  Institut  f\"{u}r Mathematik\\
  Unter den Linden 6, 10099, Berlin Germany}
\email[]{daniele.agostini@math.hu-berlin.de}

\begin{document}

\begin{abstract}
We show that vanishing of asymptotic $p$-th syzygies implies $p$-very ampleness for line bundles on arbitrary projective schemes. For smooth surfaces we prove that the converse holds when $p$ is small, by studying the Bridgeland-King-Reid-Haiman correspondence for tautological bundles on the Hilbert scheme of points. This extends previous results of Ein-Lazarsfeld, Ein-Lazarsfeld-Yang and gives a partial answer to some of their questions. As an application of our results, we show how to use syzygies to bound the irrationality of a variety.  
\end{abstract}

\maketitle

We work over the field of complex numbers. If $X$ is a projective scheme and $L$ a line bundle on $X$, we will write $L\gg 0$ if $L$ has the form $L=P\otimes A^{\otimes d}$, where $P$ is an arbitrary line bundle, $A$ an ample line bundle and $d\gg 0$. 

\section{Introduction}

Let $X$ be a smooth projective variety and $L$ an ample and globally generated line bundle: this gives a map $\phi_L\colon X\to \mathbb{P}(H^0(X,L))$ and we can regard the symmetric algebra $S=\operatorname{Sym}^{\bullet}H^0(X,L)$ as the ring of coordinates of $\mathbb{P}(H^0(X,L))$. For any line bundle $B$ on $X$ we can form a finitely generated graded $S$-module
\begin{equation}\label{modulesections} 
\Gamma_{X}(B,L) :\df \bigoplus_{q\in \mathbb{Z}} H^0(X,B\otimes L^{\otimes q}) 
\end{equation}
and then take its \textit{minimal free resolution}. It is a canonical exact complex of graded $S$-modules
\begin{equation}\label{resolution}
0 \to F_s \to F_{s-1} \to \dots \to F_1 \to F_0 \to \Gamma_X(B,L) \to 0 
\end{equation}
where the $F_i$ are free graded $S$-modules of finite rank. Taking into account the various degrees, we have a decomposition
\begin{equation}\label{koszul} 
F_p = \bigoplus_{q\in \mathbb{Z}}K_{p,q}(X,B,L)\otimes_{\mathbb{C}} S(-p-q)  
\end{equation}
for some vector spaces $K_{p,q}(X,B,L)$, called \textit{syzygy groups} or  \textit{Koszul cohomology groups}.
The Koszul cohomology groups carry a great amount of algebraic and geometric information, and they have been widely studied \cite{green, eisenbud,aprodu_nagel,ein_lazarsfeld_survey}.
\vspace{8pt}

A famous open problem in this field was the \textit{Gonality Conjecture} of Green and Lazarsfeld \cite{green_lazarsfeld}. It asserts that one can read the gonality of a smooth curve $C$ off the syzygies $K_{h^0(C,L)-2-p,1}(C,\mathcal{O}_C,L)$, for $L\gg 0$. This conjecture was confirmed for curves on Hirzebruch surfaces \cite{AproduVanishing2002} and on certain toric surfaces \cite{KawaguchiGonalityToric2008}. Most importantly, it was proven for general curves by Aprodu and Voisin \cite{AproduVoisinGonalityLarge2003} and Aprodu \cite{AproduGonalityOdd2004}. However, the conjecture for an arbitrary curve was left open, until Ein and Lazarsfeld recently gave a surprisingly quick proof \cite{ein_lazarsfeld}, drawing on Voisin's interpretation of Koszul cohomology through the Hilbert scheme \cite{voisin_even}.

More precisely,  Ein and Lazarsfeld's result is a complete characterization of the vanishing of the asymptotic $K_{p,1}(C,B,L)$ in terms of $p$\textit{-very ampleness}. If $B$ is a line bundle on a smooth projective curve $C$, we say that $B$ is \textit{$p$-very ample} if for every effective divisor $\xi\subseteq C$ of degree $p+1$, the evaluation map
\begin{equation}\label{evaluationmapintroduction}
\operatorname{ev}_{\xi}\colon H^0(C,B)\to H^0(C,B\otimes \mathcal{O}_{\xi})
\end{equation}
is surjective. Ein and Lazarsfeld proved the following \cite[Theorem B]{ein_lazarsfeld}:
\begin{equation}\label{einlazarsfeld}
K_{p,1}(C,B,L)=0 \quad \text{ for  } L\gg 0 \qquad \text{ if and only if } \qquad B \text{ is } p\text{-very ample}.
\end{equation}
In particular, this implies the Gonality Conjecture: indeed, the group $K_{h^0(C,L)-p-2,1}(C,\mathcal{O}_C,L)$ is dual to  $K_{p,1}(C,\omega_C,L)$ and Riemann-Roch shows that a curve $C$ has gonality at least $p+2$ if and only if $\omega_C$ is $p$-very ample.
\vspace{8pt}

It is then natural to wonder about an extension of (\ref{einlazarsfeld}) in higher dimensions and this was explicitly asked by Ein and Lazarsfeld in \cite[Problem 4.12]{ein_lazarsfeld_survey} and by Ein, Lazarsfeld and Yang in \cite[Remark 2.2]{ein_lazarsfeld_yang}. However, it is not a priori obvious how to generalize the statement, because the concept of $p$-very ampleness on curves can be extended to higher dimensions in at least three different ways, introduced by Beltrametti, Francia and Sommese in \cite{beltrametti_francia_sommese}.

The first one is by taking essentially the same definition: a line bundle $B$ on a projective scheme $X$ is \textit{$p$-very ample} if for every finite subscheme $\xi\subseteq X$ of length $p+1$, the evaluation map
\begin{equation}\label{evmapX}
\operatorname{ev}_{\xi}\colon H^0(X,B) \to H^0(X,B\otimes \mathcal{O}_\xi)
\end{equation}
is surjective. If instead we require that the evaluation map $\operatorname{ev}_{\xi}$ is surjective only for curvilinear schemes, the line bundle $B$ is said to be \textit{$p$-spanned}. Recall that a finite subscheme $\xi\subseteq X$ is \textit{curvilinear} if it is locally contained in a smooth curve, or, more precisely, if $\dim T_P \xi \leq 1$ for all $P\in \xi$.
The third extension is the stronger concept of jet very ampleness: a line bundle $B$ on a projective scheme $X$ is called \textit{$p$-jet very ample} if for every zero cycle $\zeta=a_1x_1+\dots+a_rx_{r}$ of degree $p+1$ the evaluation map
\begin{equation}\label{jetcondition}
\operatorname{ev}_{\zeta}\colon H^0(X,B) \to H^0(X,B\otimes \mathcal{O}_X/\mathfrak{m}_{\zeta}), \qquad \mathfrak{m}_{\zeta} :\df \mathfrak{m}_{x_1}^{a_1}\dots \mathfrak{m}_{x_r}^{a_r}
\end{equation}
is surjective.

It is straightforward to show that $p$-jet very ampleness implies $p$-very ampleness, which in turn implies $p$-spannedness. Moreover, these three concepts coincide on smooth curves, but this is not true anymore in higher dimensions: for arbitrary varieties, they coincide only when $p=0$ or $1$, and they correspond to the usual notions of global generation and very ampleness. Instead, jet very ampleness is stronger than very ampleness as soon as $p\geq 2$  \cite[Theorem p. 18]{bauer_dirocco_szemberg}.

\vspace{8pt}
The question is how these notions of higher order embeddings relate to the asymptotic vanishing of syzygies. More precisely, we want to know whether one of these notions is the correct one to generalize Ein and Lazarsfeld's result (\ref{einlazarsfeld}) for curves. This was addressed by Ein, Lazarsfeld and Yang in \cite{ein_lazarsfeld_yang}. They prove in \cite[Theorem B]{ein_lazarsfeld_yang} that if $X$ is a smooth projective variety and $K_{p,1}(X,B,L)=0$ for $L\gg 0$,  then the evaluation map $\operatorname{ev}_{\xi}\colon H^0(X,B)\to H^0(X,B\otimes \mathcal{O}_{\xi})$ is surjective for all finite subschemes $\xi \subseteq X$ consisting of $p+1$ distinct points.
For the converse, they prove in \cite[Theorem A]{ein_lazarsfeld_yang}, that if $B$ is $p$-jet very ample, then $K_{p,1}(X,B,L)=0$ for $L\gg 0$. In particular, it follows that there is a perfect analog of (\ref{einlazarsfeld}) in higher dimensions and $p=0,1$. However, it is not clear from this whether the statement should generalize to higher $p$, since in the range $p=0,1$ spannedness, very ampleness and jet very ampleness coincide.

\vspace{8pt}
Our first main  theorem is that one implication of (\ref{einlazarsfeld}) for curves  generalizes in any dimension with $p$-very ampleness, even for singular varieties. Indeed, the result holds for an arbitrary projective scheme. In particular, this strengthens \cite[Theorem B]{ein_lazarsfeld_yang}. Moreover, we can also give an effective result in the case of $p$-spanned line bundles. 

\begin{namedthm*}{Theorem A}\label{thmA}
Let $X$ be a  projective scheme and $B$ a line bundle on $X$,
\begin{equation*}\label{eqthmA}
\text{if} \quad  K_{p,1}(X,B,L) =0 \quad \text{ for } L\gg 0 \qquad \text{ then } \qquad B \text{ is } p \text{-very ample}.
\end{equation*}
 Moreover, suppose that $X$ is smooth and irreducible of dimension $n$ and let $L$ be a line bundle of the form
 \begin{equation*}
L = \omega_{X}\otimes A^{\otimes d} \otimes P ^{\otimes (n-1)} \otimes N, \qquad d\geq (n-1)(p+1)+p+3,
\end{equation*}
where $A$ is a very ample line bundle, $P$ a globally generated line bundle such that $P\otimes B^{\vee}$ is nef and $N$ a nef line bundle such that $N\otimes B$ is nef. For such a line bundle, it holds that
\begin{equation*}
\text{if} \quad  K_{p,1}(X,B,L) =0 \quad  \qquad \text{ then } \qquad B \text{ is } p \text{-spanned}.
\end{equation*}
 \end{namedthm*}
 
Our second main theorem is that on smooth surfaces we have a perfect analog of the situation (\ref{einlazarsfeld}) for curves, at least when $p$ is small.
In particular, this extends the results of \cite{ein_lazarsfeld,ein_lazarsfeld_yang}.
\begin{namedthm*}{Theorem B}\label{thmB}
Let $X$ be a smooth and irreducible projective surface, $B$ a line bundle  and $0\leq p\leq 3$ an integer: 
\begin{equation*}\label{eqhmB}
K_{p,1}(X,B,L)=0 \quad \text{ for } L\gg 0 \qquad \text{ if and only if } \qquad B \text{ is } p\text{-very ample}.
\end{equation*}
\end{namedthm*}
%The new cases here are $p=2$ and $p=3$, since the others were considered in \cite{ein_lazarsfeld_yang}.
\vspace{8pt}

As an application of these results, we generalize part of  the Gonality Conjecture to higher dimensions. More precisely, we show how to use syzygies to bound some measures of irrationality discussed recently by Bastianelli, De Poi, Ein, Lazarsfeld and Ullery \cite{irrationality}. If $X$ is an irreducible projective variety, the \textit{covering gonality} of $X$ is the minimal gonality of a curve $C$ passing through a general point of $X$. Instead, the \textit{degree of irrationality} of $X$ is the minimal degree of a dominant rational map $f\colon X \dasharrow \mathbb{P}^{\,\dim X}$. Our result is the following.

\begin{namedthm*}{Corollary C}
Let $X$ be a smooth and irreducible projective variety of dimension $n$ and suppose that $K_{h^0(X,L)-1-n-p,n}(X,\mathcal{O}_X,L)$ vanishes for $L\gg 0$. Then the covering gonality and the degree of irrationality of $X$ are at least $p+2$. 
\end{namedthm*}
In addition, we show in Corollary \ref{corgonality} that  it is enough to check the syzygy vanishing of Corollary C for a single line bundle $L$ in the explicit form of Theorem A.  Since syzygies are explicitly computable, this gives in principle an effective way to bound the  irrationality of a variety, using for example a computer algebra program.

\vspace{8pt}
Let us now describe our strategy.  We prove the first part of Theorem A by essentially reducing to the case of points in projective space. The same argument, coupled with a vanishing result of Ein and Lazarsfeld for Koszul cohomology \cite[Theorem 2]{ein_lazarsfeld_effective_vanishing}, gives also the effective result about spanned line bundles. 

For Theorem B we follow the strategy of Ein and Lazarsfeld for curves, working on the Hilbert scheme of points. The additional difficulty for a surface $X$ is that the Hilbert scheme of points $X^{[n]}$ does not coincide with the symmetric product $X^{(n)}$. We proceed to study more closely the Hilbert-Chow morphism $\mu\colon X^{[n]}\to X^{(n)}$ and we get in Proposition \ref{criterion} a characterization of the asymptotic vanishing of $K_{p,1}(X,B,L)$, purely in terms of $B$.  We show in Proposition \ref{condition} that a $p$-very ample line bundle $B$ satisfies this criterion, assuming some cohomological vanishings about the Hilbert-Chow morphism.

The key step is to prove these vanishings: we interpret them in the light of the Bridgeland-King-Reid correspondence for $X^{[n]}$, introduced by Haiman \cite{haiman_2} and further developed by Scala \cite{scala} and Krug \cite{krug,krug_mckay}. We remark that Yang has already used this correspondence to study Koszul cohomology in \cite{yang}.
With these tools, we are able to verify the desired vanishing statements for $p$ at most $3$, proving Theorem B. It may well be possible that these conditions also hold for higher $p$, but they become increasingly harder to check. We include some comments about this at the end of the article.  

Corollary C follows from Theorem A, together with  an observation about duality for Koszul cohomology and results of Bastianelli et al. \cite{irrationality}.

\vspace{8pt}
\textbf{Acknowledgments:} I am very grateful to Victor Lozovanu and Alex K\"uronya for various conversations about syzygies and surfaces. In particular, I am especially grateful to Victor Lozovanu for a discussion about the proof of Proposition \ref{condition}. I would like to warmly thank Andreas Krug for discussions on the Hilbert scheme of points, in particular about Lemma \ref{vanishings}, for his many useful comments and for pointing out a mistake in Lemma 5.1 of the first version of this paper. I am thankful to Edoardo Ballico, Mauro Beltrametti, Michael Kemeny, Yeongrak Kim, Robert Lazarsfeld, Tomasz Szemberg, Fabio Tonini and Ruijie Yang for their helpful observations and suggestions. This work is part of my PhD studies and I am indebted to my advisor Gavril Farkas for his invaluable advice and support. I would like to thank the Department of Mathematics of Stony Brook University,  for the excellent conditions provided while visiting there.
 I am supported by the DAAD, the DFG Graduiertenkolleg 1800, the DFG Schwerpunkt 1489 and the Berlin Mathematical School.

\section{Background on minimal free resolutions and Koszul cohomology}

We collect here some results on Koszul cohomology and minimal free resolutions. Good references about this are \cite{green, eisenbud, aprodu_nagel}.

The general setting is as follows: let $V$ be a vector space of finite dimension $\dim V = r+1$ and let $S=\operatorname{Sym}^{\bullet}V$ be the symmetric algebra over $V$, with its natural grading. Let $M$ be a finitely generated, graded $S$-module: then Hilbert's Syzygy Theorem asserts that there exists a unique \textit{minimal free resolution}. It is an exact complex
\begin{equation}
0 \longrightarrow F_{r+1} \longrightarrow  \dots \longrightarrow F_1 \longrightarrow F_0 \longrightarrow M \longrightarrow 0 
\end{equation}
where the $F_p$ are graded, free $S$-modules of the minimal possible rank. We can write
\begin{equation}
 F_p = \bigoplus_{q\in\mathbb{Z}} S(-p-q)\otimes_{\mathbb{C}} K_{p,q}(M;V)
\end{equation}
for certain vector spaces $K_{p,q}(M;V)$ called the \textit{Koszul cohomology groups} or \textit{syzygy groups} of $M$ w.r.t. $V$. A fundamental result about the group $K_{p,q}(M;V)$ is that it can be computed as the middle cohomology of the \textit{Koszul complex}:
\begin{equation}\label{koszulcomplex}
\wedge^{p+1}V \otimes M_{q-1} \overset{d_{p+1,q-1}}{\longrightarrow} \wedge^pV \otimes M_q \overset{d_{p,q}}{\longrightarrow} \wedge^{p-1}V\otimes M_{q+1} 
\end{equation}
where the differentials are given by
\begin{small}
\begin{equation}
d_{p,q}\colon \wedge^pV \otimes M_q \to \wedge^{p-1}V\otimes M_{q+1}, \qquad v_1\wedge \dots \wedge v_p \otimes m \mapsto \sum_{i=1}^p(-1)^{p+1}v_1\wedge \dots \wedge \widehat{v_i} \wedge \dots \wedge v_p \otimes v_i\cdot m
\end{equation}
\end{small}
We will need later the following elementary fact. We include a proof for completeness.
\begin{lemma}\label{stupidnonvanishing}
  Let $V$ be a vector space of dimension $\dim V=r+1$ and let $N=\bigoplus_{q\geq 0}N_q$ be a graded, finitely generated $\operatorname{Sym}^{\bullet}V$-module such that 
  \begin{equation}
   (0:_{N_0} V) :\df  \{ y \in N_0 \,|\, v\cdot y=0 \text{ for all } v\in V \} = 0.
 \end{equation}
 Then for any submodule $M\subset N$ such that $M_0 \subsetneq N_0$ and $M_1=N_1$, we have $K_{r,1}(M;V)\ne 0$.  
\end{lemma}
\begin{proof}
  We have a short exact sequence of $\operatorname{Sym}^{\bullet}V$-modules
  \begin{equation}
    0 \to M \to N \to N/M \to 0
  \end{equation}
  which induces a long exact sequence in Koszul cohomology (see \cite[Corollary (1.d.4)]{green}):
  \begin{equation}
 \dots \to K_{r+1,0}(N;V) \to K_{r+1,0}(N/M;V) \to K_{r,1}(M;V) \to \dots  
\end{equation}
Thanks to our hypotheses on $M$, the Koszul complex (\ref{koszulcomplex}) shows immediately that $K_{r+1,0}(N/M;V) \cong \wedge^{r+1}V \otimes (N/M)_0 \ne 0$. To conclude it suffices to show that $K_{r+1,0}(N;V)= 0$: the Koszul complex (\ref{koszulcomplex}) shows that
\begin{equation}
K_{r+1,0}(N;V) = \operatorname{Ker} \left[ d_{r+1,0}\colon \wedge^{r+1}V\otimes N_0 \to \wedge^r V\otimes N_1 \right].
\end{equation}
Now fix a basis $X_0,\dots,X_r$ of $V$: for every $y\in N_0$ we have 
\begin{equation}
d_{r+1,0}(X_0\wedge\dots \wedge X_{r+1} \otimes y) = \sum_{i=0}^{r+1}(-1)^i X_0\wedge \dots \wedge \widehat{X_i} \wedge \dots \wedge X_r \otimes X_i\cdot y
\end{equation}
hence, $d_{r+1,0}(X_0\wedge\dots \wedge X_{r+1} \otimes y)=0$ if and only if $X_i\cdot y =0$ for all $i$. But by hypothesis this implies $y=0$ and we are done.
\end{proof}

\subsection{Koszul cohomology in geometry}

The Koszul cohomology groups in the Introduction can be seen in the algebraic setting as follows: let $X$ be a projective scheme and $L$ an ample and globally generated line bundle on $X$. Then for every coherent sheaf $\mathcal{F}$ on $X$ without associated closed points, the module of sections $\Gamma_{X}(\mathcal{F},L)$ is a finitely generated and graded $\operatorname{Sym}^{\bullet}H^0(X,L)$-module. Hence, we set:
\begin{equation}
K_{p,q}(X,\mathcal{F},L) :\df K_{p,q}(\Gamma_X(\mathcal{F},L);H^0(X,L)).
\end{equation}
Moreover, even when the module of sections $\Gamma_{X}(\mathcal{F},L)$ is not finitely generated, we can define the Koszul cohomology group $K_{p,q}(X,\mathcal{F},L)$ as the middle cohomology of the Koszul complex (\ref{koszulcomplex}).

In this geometric situation one can compute Koszul cohomology via kernel bundles: since $L$ is globally generated, we have an exact sequence
\begin{equation}\label{kernelbundles}
 0 \to M_L \to H^0(X,L)\otimes_{\mathbb{C}} \mathcal{O}_X \to L \to 0
\end{equation}
which defines a vector bundle $M_L$. By a well-known result of Lazarsfeld, the above exact sequence can be used to compute Koszul cohomology, see e.g. \cite[Remark 2.6]{aprodu_nagel}:

\begin{prop}[Lazarsfeld]\label{koszulkernel}
  With the above notation, we have:
  \begin{small} 
  \begin{align}
    K_{p,q}(X,\mathcal{F},L) & \cong \operatorname{Coker} \left[ \wedge^{p+1}H^0(X,L)\otimes H^0(X,\mathcal{F}\otimes L^{\otimes (q-1)}) \to H^0(X,\wedge^p M_L \otimes \mathcal{F} \otimes L^{\otimes q}) \right] \\
    & = \operatorname{Ker } \left[ H^1(X,\wedge^{p+1}M_L\otimes L^{\otimes (q-1)} \otimes \mathcal{F}) \to \wedge^{p+1}H^0(X,L)\otimes H^1(X,L^{\otimes (q-1)})\otimes \mathcal{F}) \right].
  \end{align}
  \end{small} 
\end{prop}

Now we are going to prove a simple result about Koszul cohomology that we will need in the proof of Corollary C.

\subsection{A remark on duality for Koszul cohomology}

We first show that with some cohomological vanishings we can get a bit more from Proposition \ref{koszulkernel}:

\begin{lemma}\label{koszulhigherMl}
  With the same notation as before, fix $h>0$ and suppose that
  \begin{align}
    H^i(X,\mathcal{F}\otimes L^{\otimes (q-i)}) &= 0 & &  \text{ for all } i=1,\dots,h. \\
    H^i(X,\mathcal{F}\otimes L^{\otimes (q-i-1)}) &=0 & &  \text{ for all }  i=1,\dots,h-1.
  \end{align}
  Then
  \begin{equation}
   K_{p,q}(X,\mathcal{F},L) \cong H^h(X,\wedge^{p+h}M_L \otimes \mathcal{F} \otimes L^{\otimes (q-h)}). 
  \end{equation} 
\end{lemma}
\begin{proof}
  We proceed by induction on $h$. If $h=1$ the statement follows immediately from Proposition \ref{koszulkernel}. If instead $h>1$, taking exterior powers in the exact sequence (\ref{kernelbundles}) and tensoring by $\mathcal{F}\otimes L^{(q-h)}$ we get an exact sequence
  \begin{small} 
  \begin{equation}
   0 \to \wedge^{p+h}M_L \otimes L^{\otimes (q-h)} \otimes \mathcal{F} \to \wedge^{p+h}H^0(X,L)\otimes L^{\otimes (q-h)}\otimes \mathcal{F} \to \wedge^{p+h-1}M_L \otimes L^{\otimes (q-h+1)} \otimes \mathcal{F} \to 0.
 \end{equation}
 \end{small} 
 The statement follows from the induction hypothesis by taking the exact sequence in cohomology.
\end{proof}

Using this lemma, we can prove a small variant of  the Duality Theorem for Koszul cohomology. 

\begin{prop}\label{partialduality}
  Let $X$ be a smooth variety of dimension $n$, $L$ an ample and globally generated line bundle and $E$ a vector bundle such that
  \begin{align}
    H^i(X,E\otimes L^{\otimes (q-i-1)}) &= 0 & &\text{ for all } i=1,\dots,n-1. \\
    H^i(X,E\otimes L^{\otimes (q-i)}) &=0 & &\text{ for all } i=2,\dots,n-1.
  \end{align}
  Then
  \begin{equation}\label{partialdualityeq}
   \dim K_{p,q}(X,E,L) \leq \dim K_{h^0(X,L)-1-n-p,n+1-q}(X,\omega_X\otimes E^{\vee},L).
  \end{equation}
\end{prop}
\begin{proof}
  Observe that with the additional vanishing $H^1(X,E\otimes L^{\otimes (q-1)})=0$, the two Koszul cohomology groups in (\ref{partialdualityeq}) would be dual to each other thanks to Green's Duality Theorem \cite[Theorem 2.c.6]{green}. However, the weaker result that we are after follows without that hypothesis.

  More precisely, by Proposition \ref{koszulkernel}, we know that
  \begin{equation}\dim K_{p,q}(X,E,L) \leq \dim H^1(X,\wedge^{p+1}M_L\otimes L^{\otimes (q-1)} \otimes E).\end{equation}
Using Serre's duality, we get
  \begin{align}
    H^1(X,\wedge^{p+1}M_L\otimes L^{\otimes (q-1)} \otimes E)^{\vee} & \cong H^{n-1}(X,\wedge^{p+1}M_L^{\vee}\otimes L^{\otimes(1-q)}\otimes \omega_X\otimes E^{\vee}) \\
    & \cong H^{n-1}(X,\wedge^{r-p-1}M_L\otimes L^{\otimes(2-q)}\otimes \omega_X\otimes E^{\vee})  
  \end{align}
  where in the last isomorphism we have used that $M_L$ is a vector bundle of rank $r=h^0(X,L)-1$ and determinant $\wedge^r M_L \cong L^{\vee}$. To conclude, it is enough to observe that by Serre's duality our hypotheses are the same as the vanishing conditions of Lemma \ref{koszulhigherMl}, so that we have
  \begin{equation}
   H^{n-1}(X,\wedge^{r-p-1}M_L\otimes L^{\otimes(2-q)}\otimes \omega_X\otimes E^{\vee}) \cong K_{h^0(X,L)-1-n-p,n+1-q}(X,\omega_X\otimes E^{\vee},L). 
  \end{equation}
\end{proof}

\section{Asymptotic syzygies and finite subschemes}

In this section we prove Theorem A from the Introduction. 

\begin{lemma}\label{basiclemma}
Let $X$ be a projective scheme, $B$ a line bundle on $X$ and $\xi\subseteq X$ a finite subscheme of length $p+1$ such that the evaluation map
\begin{equation}
\operatorname{ev}_{\xi}\colon H^0(X,B) \longrightarrow H^0(X,B\otimes \mathcal{O}_{\xi})
\end{equation}
is not surjective. Let also $L$ be an ample and globally generated line bundle on $X$ such that
\begin{enumerate}
	\item $H^1(X,\mathcal{I}_\xi\otimes B \otimes L^{\otimes q})=0$ for all $q>0$.
	\item $K_{p-1,2}(X,\mathcal{I}_\xi\otimes B,L)=0$.
	\item $H^1(X,\mathcal{I}_{\xi}\otimes L)=0$. 
\end{enumerate}
Then $K_{p,1}(X,B,L)\ne 0$.	
\end{lemma}
\begin{proof}
	Consider the short exact sequence of sheaves on $X$:
	\begin{equation}
	0 \to \mathcal{I}_{\xi}\otimes B \to B \to B\otimes \mathcal{O}_{\xi} \to 0.
	\end{equation}
	Twisting by powers of $L$ and taking global sections, we get an exact sequence of graded $\operatorname{Sym}^{\bullet} H^0(X,L)$-modules
	\begin{equation}\label{sequenceM}
	0 \to \bigoplus_{q\geq 0}H^0(X,\mathcal{I}_{\xi}\otimes B \otimes L^{\otimes q}) \to \bigoplus_{q\geq 0}H^0(X,B \otimes L^{\otimes q}) \to M \to 0.
	\end{equation}
	Moreover, assumption (1) shows that $M$ is a submodule of $\bigoplus_{q\geq 0} H^0(X,B\otimes L^{\otimes q}\otimes \mathcal{O}_{\xi})$ such that
	\begin{equation}\label{structureM}
	M_0 \subsetneq H^0(X,B\otimes \mathcal{O}_{\xi}), \qquad M_q=H^0(X,B\otimes L^{\otimes q}\otimes \mathcal{O}_{\xi}) \qquad \text{ for all } q>0. 
	\end{equation}
	The sequence (\ref{sequenceM}) induces an exact sequence in Koszul cohomology \cite[Corollary (1.d.4)]{green}
	\begin{equation}
	\dots \to K_{p,1}(X,B,L) \to K_{p,1}(M;H^0(X,L)) \to K_{p-1,2}(X,\mathcal{I}_{\xi}\otimes B,L) \to \dots
	\end{equation}
	and using assumption (2) we get that the natural map $K_{p,1}(X,B,L) \to K_{p,1}(M;H^0(X,L))$ is surjective. Hence, it is enough to show that $K_{p,1}(M;H^0(X,L))\ne 0$.

	To do this, observe that the structure of $\operatorname{Sym}^{\bullet}H^0(X,L)$-module on $M$ is induced by the structure as $\operatorname{Sym}^{\bullet}H^0(X,L\otimes \mathcal{O}_{\xi})$-module. Moreover, assumption (3) shows that the evaluation map $H^0(X,L)\longrightarrow H^0(X,L\otimes\mathcal{O}_\xi)$ is surjective. Hence, a standard argument for the computation of Koszul cohomology w.r.t.  different rings (see e.g. \cite[Lemma 2.1]{agostini_kuronya_lozovanu}) produces a decomposition
	\begin{equation}
	K_{p,1}(M;H^0(X,L)) \cong \bigoplus_{i=0}^p \wedge^{p-i}H^0(X,\mathcal{I}_{\xi} \otimes L) \otimes K_{i,1}(M;H^0(X,L\otimes \mathcal{O}_{\xi})).
	\end{equation}
	To conclude, the description of $M$ in (\ref{structureM}) and Lemma \ref{stupidnonvanishing}, give  $K_{p,1}(M;H^0(C,L\otimes \mathcal{O}_{\xi})) \ne 0$ and we are done.
\end{proof}

We need a statement for the asymptotic vanishing of high degree syzygies. This is probably already known but we include a proof for completeness. 

\begin{lemma}\label{asymptoticvanishingsyzygieshighdegree}
Let $X$ be a projective scheme, $A$ an ample line bundle and $P$ an arbitrary line bundle on $X$. For any  integer $d>0$ set $L_d = A^{\otimes d}\otimes P$. Fix a coherent sheaf $\mathcal{F}$ on $X$ and two integers $p\geq 0,q\geq 2$.
Then $K_{p,q}(X,\mathcal{F},L_d)=0$ for infinitely many $d$.
\end{lemma}
\begin{proof}
  First suppose that $X$ is smooth. In this case we claim that $K_{p,q}(X,\mathcal{F},L_d)=0$ for all $d\gg 0$. If $\mathcal{F}$ is locally free, we have $K_{p,q}(X,\mathcal{F},L_d)=0$ for $d\gg 0$, thanks for example to \cite[Proof of Theorem 4]{yang}. Assume now that $\mathcal{F}$ is an arbitrary coherent sheaf. Since $X$ is smooth, $\mathcal{F}$ has a finite resolution by locally free sheaves: we can choose a resolution with the minimum length $\ell$, so that we get an exact complex
\begin{equation}
0 \to E_{\ell} \to E_{\ell-1} \to \dots \to E_0 \to \mathcal{F} \to 0,
\end{equation}
where the $E_i$ are locally free. We proceed to prove the lemma by induction on $\ell$. If $\ell=0$ then $\mathcal{F}$ is locally free and we are done. If $\ell>0$, we can split the resolution into two exact complexes
\begin{align}
   0 \to \mathcal{G} \to &E_0 \to \mathcal{F} \to 0, \\
  0 \to E_{\ell} \to E_{\ell-1} \to &\dots \to E_1 \to \mathcal{G} \to 0. 
\end{align}
Since $d\gg 0$, we get $H^1(X,\mathcal{G}\otimes L^{\otimes q}_d) = 0$ for all $q\geq 1$, so that we obtain a short exact sequence of $\operatorname{Sym}^{\bullet} H^0(X,L_d)$-graded modules:
\begin{equation}
0 \to \bigoplus_{h\geq 1} H^0(X,\mathcal{G}\otimes L^{\otimes h}_d) \to \bigoplus_{h\geq 1} H^0(X,E_0 \otimes L_d^{\otimes h}) \to  \bigoplus_{h\geq 1}H^0(X,\mathcal{F}\otimes L_d^{\otimes h}) \to 0
\end{equation}
Since $q\geq 2$, this sequence  induces an exact sequence in Koszul cohomology \cite[Corollary (1.d.4)]{green}:
\begin{equation}
\dots \to K_{p,q}(X,E_0,L_d) \to K_{p,q}(X,\mathcal{F},L_d) \to K_{p-1,q+1}(X,\mathcal{G},L_d) \to \dots
\end{equation}
If $d\gg 0$ we know that $K_{p,q}(X,E_0,L_d)=0$ because $E_0$ is locally free.  Moreover, $K_{p-1,q+1}(X,\mathcal{G},L_d)=0$ by induction hypothesis. Hence, $K_{p,q}(X,\mathcal{F},L_d)=0$ as well, and we are done.

Now take an arbitrary projective scheme $X$. We claim that it is enough to find a closed embedding $j\colon X \hookrightarrow Y$ such that $Y$ is smooth and it has two line bundles $\widetilde{A},\widetilde{P}$, with $\widetilde{A}$ ample,  such that $j^*\widetilde{A} \cong A$ and $j^*\widetilde{P} \cong P$.  Indeed, in this case set $\widetilde{L}_d = \widetilde{P}\otimes \widetilde{A}^{\otimes d}$: if $d\gg 0$ we can assume that the restriction map
\begin{equation}\label{restrictionmapinlemmavanishing}
H^0(Y,\widetilde{L}_d) \to H^0(X,L_d)
\end{equation}
is surjective.
Since $Y$ is smooth, what we have already proved shows that $K_{p,q}(Y,j_*\mathcal{F},\widetilde{L}_d) = 0$ for $d\gg 0$. However, the structure of $\operatorname{Sym}^{\bullet}H^0(Y,L_d)$-module on
\begin{equation}
\bigoplus_h H^0(Y,j_*\mathcal{F}\otimes \widetilde{L}^{\otimes h}_d) = \bigoplus_h H^0(X,\mathcal{F}\otimes L_d^{\otimes h}) 
\end{equation}
is induced by the structure of $\operatorname{Sym}^{\bullet} H^0(X,L_d)$-module via the map (\ref{restrictionmapinlemmavanishing}). Hence, using a standard result on Koszul cohomology w.r.t. two different rings \cite[Lemma 2.1]{agostini_kuronya_lozovanu}, we see that $K_{p,q}(X,\mathcal{F},L_d)=0$ as well.

Now, we just need to find the embedding $j\colon X \hookrightarrow Y$. Observe that in the original statement we can replace $P$ by a translate $P\otimes A^{\otimes h}$, and $A$ by a positive multiple $A^{\otimes m}$. Hence, we choose $h,k$ positive such that both $P\otimes A^{\otimes h}$ and $P^{\vee}\otimes A^{\otimes k}$ are very ample, and consider the induced closed embedding $\varphi\colon X \hookrightarrow \mathbb{P}^n \times \mathbb{P}^m$. Then we see that $\varphi^*\mathcal{O}(1,0) = P \otimes A^{\otimes h}$, $\varphi^*\mathcal{O}(1,1) =  A^{\otimes(h+k)}$.  Since $\mathcal{O}(1,1)$ is ample, we are done.
\end{proof}

With this we could already give the proof of the first part of Theorem A, but we postpone it until the end of the next section, so that we can also prove the second part.

\subsection{An effective result for spanned line bundles}\label{subsectioneffective}

Here we give a proof of the second part of Theorem A. The idea is to find effective bounds for the conditions of Lemma \ref{basiclemma}. The essential reason that we restrict to spannedness instead of very ampleness is to have an effective vanishing statement along the lines of Lemma \ref{asymptoticvanishingsyzygieshighdegree}: this is given by a result of Ein and Lazarsfeld \cite[Theorem 2]{ein_lazarsfeld_effective_vanishing}.

The proof is essentially by induction on the dimension of $X$ and it is based on the next lemmas.
A word about notation: if $Y\subseteq X$ is an inclusion of varieties and if $L$ is a line bundle on $X$, we denote by $L_Y$ the restriction of $L$ to $Y$.

\begin{lemma}\label{lemmaonedivisor}
Let $X$ be a smooth projective variety, $L$ an ample and globally generated line bundle on $X$, $B$ another line bundle and $p\geq 0$ an integer. Let also $D\subseteq X$ be a divisor such that:
\begin{enumerate}
\item $H^1(X,L^{\otimes q} \otimes B \otimes \mathcal{O}_X(-D)) = 0$\, for all $q\geq 0$.
\item $K_{p-1,2}(X,B\otimes\mathcal{O}_X(-D),L)=0$.
\item $H^1(X,L\otimes \mathcal{O}_X(-D)) = 0$.
\end{enumerate}
Then the natural maps
\begin{equation} 
H^0(X,B) \to H^0(D,B_D), \qquad K_{p,1}(X,B,L) \to K_{p,1}(D,B_D,L_D) 
\end{equation}
are surjective.
\end{lemma}
\begin{proof}
The proof goes along the same lines as that of Lemma \ref{basiclemma}, so we give here just a sketch. Hypothesis (1) gives  a short exact sequence of graded $\operatorname{Sym}^{\bullet}H^0(X,L)$-modules:
\begin{equation}
 0 \to \bigoplus_{q\geq 0}H^0(X,L^{\otimes q} \otimes B \otimes \mathcal{O}_X(-D)) \to \bigoplus_{q\geq 0} H^0(X,B\otimes L^{\otimes q}) \to M \to 0
\end{equation}
where $M = \bigoplus_{q\geq 0}H^0(D,L_D^{\otimes q} \otimes  B_{D})$. The long exact sequence in Koszul cohomology and hypothesis (2) show that the natural map
\begin{equation}
K_{p,1}(X,B,L) \to K_{p,1}(M;H^0(X,L))
\end{equation}
is surjective. Using hypothesis (3) and a standard argument for the computation of Koszul cohomology w.r.t.  different rings we get a natural surjective map 
\begin{equation}
K_{p,1}(M;H^0(X,L)) \to  K_{p,1}(D,B_D,L_D).
\end{equation}
In particular, the composite map $K_{p,1}(X,B,L) \to K_{p,1}(X,B_D,L_D)$ is surjective, and this is the map we were looking for. 	
\end{proof}

\begin{lemma}\label{auxiliarydivisor}
Let $X$ be a smooth and irreducible projective variety of dimension  $n\geq 2$. Let $\xi\subseteq X$ be a curvilinear subscheme of length $\ell(\xi)=k$ and $H$ an ample and $k$-jet very ample line bundle on $X$. Then there exists a smooth and irreducible divisor $D\in |H|$ such that $\xi\subseteq D$.
\end{lemma}
\begin{proof}
Consider the linear system $V = H^0(X,H\otimes \mathcal{I}_{\xi})$. We will show that a general divisor in $|V|$ is smooth and irreducible.
We first show that $V$ has base points only at the points of $\xi$. If $P\notin \xi$, the subscheme $\xi \cup \{P\}$ has length $k+1$, and since $H$ is in particular $k$-very ample, the map $\operatorname{ev}_{\xi\cup \{P\}}\colon H^0(X,H) \to H^0(X,H\otimes \mathcal{O}_{\xi\cup \{ P\}})$ is surjective.  Hence $P$ is not a base point of $V$. Now, Bertini's theorem tells us that a general divisor $D\in |V|$ is irreducible and nonsingular away from the support of $\xi$. We need to check what happens at the points in $\xi$, and for this we can suppose that $\xi$ is supported at a single point $P$. Since $\xi$ is curvilinear, we can find \cite[Remarks 2.1.7, 2.1.8]{goettsche} analytic coordinates $(x_1,\dots,x_n)$ around $P$ such that we have the local description $\mathcal{I}_{\xi} = (x_1,\dots,x_{n-1},x_n^k)$. Moreover, as $H$ is $k$-jet very ample, the map $H^0(X,H) \to H^0(X,H\otimes \mathcal{O}_X/\mathfrak{m}_P^k)$ is surjective. Hence, the power series expansion of a general section $\sigma \in V$ around $P$ has a nonzero coefficient for $x_1$, so that $\sigma$ defines a divisor which is nonsingular at $P$.
\end{proof}

Now we can start the proof of the second part of Theorem A. The first case is that of curves. 

\begin{prop}\label{effnonvanishingcurves}
	Let $C$ be a smooth, projective and irreducible curve of genus $g$, and $B$ a line bundle which is not $p$-very ample. Let also $L$ be a line bundle such that
	\begin{equation}\label{conditionsdegrees}
	\deg L \geq 2g+p+1, \qquad \deg (L\otimes B) \geq 2g+p+1.
	\end{equation}
	Then $K_{p,1}(C,B,L)\ne 0$. 
\end{prop}
\begin{proof}
	Observe that $L$ is ample and globally generated. Suppose first that $h^0(C,B)\geq p+1$. Let $\xi \subseteq C$ be an effective divisor of degree $p+1$ such that the evaluation map $\operatorname{ev}_{\xi} \colon H^0(C,B) \to H^0(C,B\otimes \mathcal{O}_{\xi})$ is not surjective. We show now that $L$ satisfies the conditions of Lemma \ref{basiclemma}. Since $\deg L \geq 2g+p$ and $B$ is effective, it is easy to see that conditions (1) and (3) hold. To check condition (2), we need to show that $K_{p,1}(C,B\otimes \mathcal{O}_C(-\xi),L)=0$. By Proposition \ref{koszulkernel} it is enough to show that
	$ H^1(C,\wedge^{p}M_L\otimes L \otimes B\otimes \mathcal{O}_C(-\xi))=0$. Since $\deg L \geq 2g+p$, a result of Green \cite[Theorem (4.a.1)]{green} gives that $H^1(C,\wedge^pM_L\otimes L)=0$. Hence, if we can prove that $B\otimes \mathcal{O}_C(-\xi)$ is effective, it follows that $ H^1(C,\wedge^{p}M_L\otimes L \otimes B\otimes \mathcal{O}_C(-\xi))=0$ as well. To check that $B\otimes \mathcal{O}_C(-\xi)$ is effective, observe that $h^0(C,B)\geq p+1$ by assumption, and moreover the evaluation map $\operatorname{ev}_{\xi}$ is not surjective, so that $h^0(C,B\otimes \mathcal{O}_C(-\xi)) > h^0(C,B)-p-1 \geq 0$, and we are done.
	
	Now assume $h^0(C,B)\leq p$. Proposition \ref{kernelbundles} gives that $K_{p,1}(C,B,L)$ is the cokernel  of the map
	\begin{equation}
	\wedge^{p+1}H^0(C,L)\otimes H^0(C,B) \to H^0(C,\wedge^p M_L \otimes L \otimes B). 
	\end{equation}
	Thus, to prove what we want it is enough to show that
	\begin{equation}\label{comparison}
	\dim \wedge^{p+1}H^0(C,L)\otimes H^0(C,B) < \dim  H^0(C,\wedge^p M_L \otimes L \otimes B).
	\end{equation} 
	To do this, set $d=\deg L$ and $b=\deg B$. We can estimate the dimension of $H^0(C,\wedge^p M_L \otimes L \otimes B)$ via the Euler characteristic, which is easy to compute with Riemann-Roch:
	\begin{equation}
	h^0(C,\wedge^p M_L \otimes L \otimes B) \geq \chi(C,\wedge^p M_L \otimes L \otimes B)  = \binom{d-g}{p}\left( - p\cdot \frac{d}{d-g} + d + b \right).
	\end{equation} 
	Now, suppose that $0<h^0(C,B) \leq p$: in particular $b\geq 0$. We can just bound the left hand side of (\ref{comparison}) by $\binom{d+1-g}{p+1}p$ and then a computation shows that (\ref{comparison}) holds, thanks to $d\geq 2g+p+1$ and $b\geq 0$.
	
	The last case is when $h^0(C,B)=0$. To prove (\ref{comparison}) it is enough to show that $\chi(C,\wedge^p M_L \otimes L \otimes B) > 0$. This can be checked by a computation, using the assumption that $d+b \geq 2g+p+1$.
\end{proof}

\begin{rmk}
  Going through the computation of Proposition \ref{effnonvanishingcurves} more carefully, it is not hard to show that the assumption on $L$ can be weakened to $\deg L \geq 2g+p$, at least when $C$ has genus $g\geq 2$. In this case, setting $B=\omega_C$, Proposition \ref{effnonvanishingcurves} gives that if $C$ has gonality $k$, then
  \begin{equation}
    K_{k-1,1}(C,\omega_C,L) \cong K_{h^0(L)-k-1,1}(C,\mathcal{O}_C,L)\ne 0
  \end{equation}
  for every line bundle $L$ of degree $\deg L \geq 2g+k-1$. This is well-known and an easy consequence of the Green-Lazarsfeld Nonvanishing Theorem \cite[Appendix]{green}. Conversely, Farkas and Kemeny proved a vanishing theorem in \cite[Theorem 0.2]{farkas_kemeny}:  if $C$ is a general $k$-gonal curve of genus at least $4$, then $K_{h^0(L)-k,1}(C,\mathcal{O}_C,L)=0$, when $\deg L \geq 2g+k-1$. However they note in the same paper that this vanishing does not hold for every curve. 
\end{rmk}

Now we can give the full proof for the second part of Theorem A: we rewrite the statement below for clarity, and we formulate it as a nonvanishing statement.

\begin{thm}\label{thmAsecondpart}
Let $X$ be a smooth and irreducible projective variety of dimension $n$, and $B$ a line bundle on $X$ which is not $p$-spanned. Then $K_{p,1}(X,B,L)\ne 0$ for every line bundle $L$ of the form
\begin{equation}
L = \omega_X\otimes A^{\otimes d}\otimes P^{\otimes(n-1)}\otimes N, \qquad d \geq (n-1)(p+1)+p+3,
\end{equation} 
where $A$ is a very ample line bundle, $P$ a globally generated line bundle such that $P\otimes B^{\vee}$ is nef, and $N$ is a nef line bundle such that $N\otimes B$ is nef.
\end{thm}
\begin{proof}
First we observe that any $L$ as in the statement of the theorem is very ample: indeed, Kodaira vanishing shows that $L\otimes A^{\vee}$ is $0$-regular w.r.t. $A$ in the sense of Castelnuovo-Mumford. In particular it is globally generated. Hence $L=(L\otimes A^{\vee})\otimes A$ is very ample.
  
Now we proceed to prove the theorem by induction on $n$. If $n=1$, set $g$ to be the genus of the curve $X$: then we see that $\deg L \geq 2g-2+d \geq 2g+p+1$, and the same holds for $\deg (L\otimes B)$. Hence, the conclusion follows from Proposition \ref{effnonvanishingcurves}.

Now, suppose that $n\geq 2$ and that the result is true for $n-1$. Fix a finite, curvilinear scheme $\xi \subseteq X$ of length $p+1$ such that the evaluation map
\begin{equation} 
\operatorname{ev}_{\xi}\colon H^0(X,B) \to H^0(X,B\otimes \mathcal{O}_{\xi})
\end{equation}
is not surjective. Consider the line bundle $H=P \otimes A^{\otimes (p+1)}$: since $P$ is globally generated and $A$ is very ample, $H$ is $(p+1)$-jet very ample (see \cite[Lemma 2.2]{beltrametti_sommese}).  Hence, Lemma \ref{auxiliarydivisor} shows that there is a smooth and irreducible divisor $D\in |H|$ such that $\xi\subseteq D$.

Now, let $L$ be as in the statement of the theorem: we claim that $L,B$ and $D$ satisfy the hypotheses of Lemma \ref{lemmaonedivisor}. Indeed, we see that
\begin{equation}
L\otimes \mathcal{O}_X(-D) \cong L \otimes H^{\vee}  \cong  \omega_X \otimes A^{\otimes (d-p-1)} \otimes P^{\otimes (n-2)} \otimes N
\end{equation}
and the assumption  on $d$ shows that $A^{\otimes (d-p-1)} \otimes P^{\otimes (n-2)} \otimes N$ is ample, so that $H^1(X,L\otimes \mathcal{O}_X(-D))=0$ by Kodaira vanishing. A similar reasoning shows that $H^1(X,L^{\otimes q} \otimes B \otimes  \mathcal{O}_X(-D))=0$ for all $q\geq 1$. To check that $H^1(X,B\otimes \mathcal{O}_X(-D))=0$, observe that $H^1(X,B\otimes \mathcal{O}_X(-D))^{\vee}\cong H^{n-1}(X,\omega_X\otimes B^{\vee} \otimes H)$ and $H\otimes B^{\vee} = P\otimes B^{\vee}\otimes A^{\otimes (p+1)}$ is clearly ample, so that we can use Kodaira vanishing again, together with the assumption $n\geq 2$.

 Finally, a result of Ein and Lazarsfeld \cite[Theorem 2]{ein_lazarsfeld_effective_vanishing} shows that  $K_{p-1,2}(X,B\otimes \mathcal{O}_X(-D),L)$ vanishes: indeed, we can write
\begin{equation}
 L \cong \omega_X \otimes A^{\otimes (n+p)} \otimes A^{\otimes (d-n-p)} \otimes  P^{\otimes(n-1)} \otimes N.   
\end{equation}
and since $d-n-p \geq (n-1)p+2$ we see that $A^{\otimes (d-n-p)} \otimes  P^{\otimes(n-1)} \otimes N$ is nef. Furthermore
\begin{equation}
A^{\otimes (d-n-p)} \otimes  P^{\otimes(n-1)} \otimes N \otimes B \otimes \mathcal{O}_X(-D) \cong A^{\otimes (d-n-2p-1)}\otimes P^{\otimes (n-2)}\otimes B\otimes N
\end{equation}
and since $d-n-2p-1\geq (n-2)p+1$, we see again that this is nef. Then the aforementioned \cite[Theorem 2]{ein_lazarsfeld_effective_vanishing} applies and, we get that $K_{p-1,2}(X,B\otimes \mathcal{O}_X(-D),L)=0$.

Now we can apply Lemma \ref{lemmaonedivisor} and we obtain that the two natural restriction maps
\begin{equation}
H^0(X,B) \to H^0(D,B_D), \qquad K_{p,1}(X,B,L) \to K_{p,1}(D,B_D,L_D)
\end{equation}
are surjective. In particular, since $\xi\subseteq D$, we see that $B_D$ is not $p$-spanned on $D$. Moreover, the adjunction formula shows that
\begin{equation}
L_D = K_D \otimes A_D^{\otimes (d-(p-1))}\otimes P_D^{\otimes (n-2)}\otimes N_D
\end{equation}
which clearly satisfies the induction hypothesis for $n-1$. Hence $K_{p,1}(D,B_D,L_D)\ne 0$, and since $K_{p,1}(X,B,L) \to K_{p,1}(D,B_D,L_D)$ is surjective, we are done.
\end{proof}

Now we can prove Theorem A.

\begin{proof}[Proof of Theorem A]
We start from the first part. Let $X$ be a projective scheme, and $B$ a line bundle on $X$. Fix also an ample line bundle $A$, another line bundle $P$ and set $L_{d} = P\otimes A^{\otimes d}$ for any integer $d>0$. Assume that $K_{p,1}(X,B,L_d)=0$ for $d\gg 0$. We want to show that $B$ is $p$-very ample. So, we assume that $B$ is not $p$-very ample and we claim that $K_{p,1}(X,B,L_d)\ne 0$ for infinitely many
$d$. 

To do this, let $\xi\subseteq X$ be a finite subscheme of length $\ell(\xi) = p+1$ such that the evaluation map
\begin{equation}
\operatorname{ev}_{\xi} \colon H^0(X,B) \to H^0(X,B\otimes \mathcal{O}_{\xi})
\end{equation}
is not surjective. Then it is enough to show that the hypotheses in Lemma \ref{basiclemma} are verified for infinitely many $d$. Hypotheses (1) and (3) hold for all $d\gg 0$ thanks to Serre vanishing. Lemma \ref{asymptoticvanishingsyzygieshighdegree} gives hypothesis (2) and we are done.

The second part of the theorem is exactly Theorem \ref{thmAsecondpart}. 
\end{proof}
 
\subsection{Asymptotic syzygies and measures of irrationality}

As an application of Theorem A we give a proof of Corollary C from the Introduction.
First we prove a related result, which extends part of \cite[Corollary C]{ein_lazarsfeld_yang}. We observe that we do not require the condition $H^i(X,\mathcal{O}_X)=0$ for $0<i<\dim X$, which is present in \cite[Corollary C]{ein_lazarsfeld_yang}.

\begin{cor}
Let $X$ be a smooth and irreducible projective variety of dimension $n$.
\begin{equation} 
\text{If} \quad  K_{h^0(X,L)-1-n-p,n}(X,\mathcal{O}_X,L) =0 \quad \text{ for } L\gg 0 \qquad \text{ then } \qquad \omega_X \text{ is } p \text{-very ample}. 
\end{equation}
\end{cor}
\begin{proof}
Since $L\gg 0$, we see that $H^{n-i}(X,L^{\otimes i})=0$ for all $i=1,\dots,n-1$ and $H^{n-i}(X,L^{\otimes (i-1)})=0$ for all $i=2,\dots,n-1$. Hence, using Serre's duality and Proposition \ref{partialduality}, we get
\begin{equation}
\dim K_{p,1}(X,\omega_X,L) \leq \dim K_{h^0(X,L)-1-n-p,n}(X,\mathcal{O}_X,L).
\end{equation}
Thus, $K_{h^0(X,L)-1-n-p,n}(X,\mathcal{O}_X,L)=0$ implies $K_{p,1}(X,\omega_X,L)=0$ as well, so that we conclude using Theorem A.	
\end{proof}

A similar proof, together with results from \cite{irrationality}, gives Corollary C. We actually give here a more precise version, which contains the effective result mentioned in the Introduction.

\begin{cor}\label{corgonality} 
Let $X$ be a smooth and irreducible projective variety of dimension $n$. Let $L$ be a line bundle of the form
\begin{equation}
L = \omega_{X}\otimes A^{\otimes d} \otimes P ^{\otimes (n-1)} \otimes N, \qquad d\geq (n-1)(p+1)+p+3,
\end{equation}
where $A$ is a very ample line bundle, $P$ a globally generated line bundle such that $P\otimes \omega_X^{\vee}$ is nef and $N$ a nef line bundle such that $N\otimes \omega_X$ is nef. If $K_{h^0(X,L)-1-n-p,n}(X,\mathcal{O}_X,L)=0$ then the covering gonality and the degree of irrationality of $X$ are at least $p+2$. 
\end{cor}
\begin{proof}
For such a line bundle $L$, Kodaira Vanishing implies that $H^{n-i}(X,L^{\otimes i})=0$ for all $i=1,\dots,n-1$ and $H^{n-i}(X,L^{\otimes (i-1)})=0$ for all $i=2,\dots,n-1$. Hence, Serre's duality and Proposition \ref{partialduality} give
\begin{equation}
\dim K_{p,1}(X,\omega_X,L) \leq \dim K_{h^0(X,L)-1-n-p,n}(X,\mathcal{O}_X,L).
\end{equation}
Thus, $K_{h^0(X,L)-1-n-p,n}(X,\mathcal{O}_X,L)=0$ implies $K_{p,1}(X,\omega_X,L)=0$ as well. Therefore, Theorem A shows that $\omega_X$ is $p$-spanned. Consider now a smooth curve $C$ and a map $f\colon C\to X$ which is birational onto its image. Then it is immediate from the definition that the line bundle $f^*\omega_X$ is birationally $p$-very ample on $C$, according to the definition of Bastianelli et al. \cite[Definition 1.1]{irrationality}. With this, a straightforward variation in the proof of \cite[Theorem 1.10]{irrationality} gives that the covering gonality of $X$ is at least $p+2$. Since the covering gonality is always smaller than the degree of irrationality \cite[(3.1) page 13]{irrationality}, this concludes the proof. 
\end{proof}	

Now we turn to the case of surfaces, with the aim of proving Theorem B. We start by recalling some facts about the Hilbert scheme of points on smooth surfaces.

\section{Background on the Hilbert scheme of points on a smooth surface}

We will collect here some results about the Hilbert scheme of points for quasiprojective surfaces. 
Let $X$ be a smooth, irreducible,  quasiprojective surface and $n>0$ a positive integer: we will denote by $X^{[n]}$ the Hilbert scheme of points of $X$ and by $X^{(n)}$ the symmetric product of $X$. The Hilbert scheme $X^{[n]}$ parametrizes finite subschemes $\xi\subseteq X$ of length $n$, whereas $X^{(n)}$ parametrizes zero cycles of length $n$ on $X$. If $X$ is projective, both $X^{[n]}$ and $X^{(n)}$ are projective as well.

The symmetric product can be obtained as the quotient $X^{(n)} = X^n/\mathfrak{S}_n$, where $\mathfrak{S}_n$ acts naturally on $X^n$. 
%by
%\begin{equation}\label{actionSn}  
%\sigma \cdot (P_1,\dots,P_n) = (P_{\sigma^{-1}(1)},\dots,P_{\sigma^{-1}(n)}) 
%\end{equation}
We denote by
\begin{equation}\label{projection} 
\pi\colon X^n \to X^{(n)} 
\end{equation}
the projection.
 There is also a canonical Hilbert-Chow morphism
\begin{equation}\label{hilbchow} 
\mu\colon X^{[n]} \to X^{(n)} \qquad \xi \mapsto \sum_{P\in X} \ell(\mathcal{O}_{\xi,P}) \cdot P 
\end{equation}
that maps a subscheme to its weighted support. By construction, the Hilbert scheme comes equipped with a universal family $\Xi^{[n]}$, that can be described as
\begin{equation}\label{univfamily} 
\Xi^{[n]} = \{ (P,\xi) \in X\times X^{[n]} \,|\, P \in \xi  \}, \qquad p_X\colon \Xi^{[n]} \to X, \qquad p_{X^{[n]}}\colon \Xi^{[n]} \to X^{[n]} 
\end{equation}
with the map $p_{X^{[n]}}$ being finite, flat and of degree $n$: the fiber of $p_{X^{[n]}}$ over $\xi\in X^{[n]}$ is precisely the subscheme $\xi\subseteq X$.

The same construction can be carried out for every quasiprojective scheme, however, when $X$ is an irreducible smooth surface, Fogarty \cite{fogarty} proved that 
 $X^{[n]}$ is a smooth and irreducible variety of dimension $2n$. Moreover the symmetric product $X^{(n)} $ is irreducible, Gorenstein, with rational singularities and the Hilbert-Chow morphism $\mu\colon X^{[n]} \to X^{(n)}$ is a crepant resolution of singularities, so that $\mu^*\omega_{X^{(n)}} \cong \omega_{X^{[n]}}$.

\begin{rmk}\label{rmkcurves}
For smooth curves, the Hilbert scheme is also smooth and irreducible. Moreover the Hilbert-Chow morphism is an isomorphism.  
\end{rmk}

We will need later an estimate on the amount of curvilinear subschemes:

\begin{rmk}\label{sizecurvilinear}
Recall that a subscheme $\xi\in X^{[n]}$ is said to be curvilinear if $\dim T_{P}\xi \leq 1$ for all $P\in X$. The set $U_n\subseteq X^{[n]}$ of curvilinear subschemes is open and dense, and its complement has codimension 4 \cite[Remark 3.5]{beltrametti_francia_sommese}.
\end{rmk}

\subsection{Tautological bundles}

If $L$ is any line bundle on $X$, the line bundle $L^{\boxtimes n} = \bigotimes_{i=1}^n pr_i^*L$ has a $\mathfrak{S}_n$-linearization. Hence, we can take the sheaf of invariants $L^{(n)} :\df \pi_*^{\mathfrak{S}_n}(L^{\boxtimes n})$ which is a coherent sheaf on $X^{(n)}$. In fact, it was proven by Fogarty \cite{fogarty_2} that $L^{(n)}$ is a line bundle on $X^{(n)}$ such that $\pi^* L^{(n)} \cong L^{\boxtimes n}$ and that the induced map
\begin{equation}\label{homomophism} 
\operatorname{Pic}(X) \to \operatorname{Pic}(X^{(n)}) \qquad L \mapsto L^{(n)} 
\end{equation}
is an homomorphism of groups. This gives a line bundle on $X^{[n]}$ by taking $\mu^*L^{(n)}$. 

Since the map $\pi\colon X^n \to X^{(n)}$ is finite, we get the following well known result:

\begin{lemma}\label{ampleness}
If $X$ is projective and $A$ is an ample bundle on $X$, then $A^{(n)}$ is ample on $X^{(n)}$. In particular, if $L \gg 0$ on $X$, then $L^{(n)}\gg 0$ on $X^{(n)}$. 
\end{lemma}

Another construction of bundles on the Hilbert scheme is the following: let $E$ be a vector bundle on $X$ of rank $r$. Then we can define the \textit{tautological bundle} associated to $E$ on $X^{[n]}$ as
\begin{equation}\label{tautological} 
E^{[n]}:\df p_{X^{[n]},*} p_X^* (E). 
\end{equation}
Since the map $p_{X^{[n]}}\colon \Xi^{[n]} \to X^{[n]}$ is finite and flat of degree $n$, the sheaf $E^{[n]}$ is a vector bundle of rank $n\cdot r$ on $X^{[n]}$. By construction, the fiber of $E^{[n]}$ over a point $\xi\in X^{[n]}$ is identified with $H^0(X,E\otimes \mathcal{O}_{\xi})$.

We can also define a line bundle on $X^{[n]}$ by
\begin{equation}\label{delta} 
\mathcal{O}(-\delta_n) :\df \det \mathcal{O}_X^{[n]}. 
\end{equation}
A geometrical interpretation of this line bundle is that the class $2\delta_n$ represents the locus of non-reduced subschemes in $X^{[n]}$, which is the exceptional divisor of the Hilbert-Chow morphism $\mu\colon X^{[n]} \to X^{(n)}$. 

The determinant of a tautological bundle is well-known:
\begin{equation}\label{dettaut}
\det L^{[n]} \cong \mu^*L^{(n)}\otimes \mathcal{O}(-\delta_n).
\end{equation}

\subsection{Derived McKay correspondence for the Hilbert scheme of points}\label{section_mckay}

For a more extensive exposition on this section, we refer to \cite{krug_mckay}.

Using the theory of Bridgeland-King-Reid \cite{BKR}, Haiman obtained in \cite{haiman_1},\cite{haiman_2} a fundamental description of the derived category $D^b(X^{[n]})$ in terms of $\mathfrak{S}_n$-linearized coherent sheaves on $X^n$. More precisely, denote by $D^b_{\mathfrak{S}_n}(X^n)$ the derived category of $\mathfrak{S}_n$-linearized coherent sheaves on $X^n$. Then Haiman's result is the following:

\begin{thm}[Haiman]\label{mckay}
There are explicit equivalences of derived categories
\begin{equation}
  \Phi\colon D^b(X^{[n]}) \to D^b_{\mathfrak{S}_n}(X^n), \qquad \Psi\colon D^b_{\mathfrak{S}_n}(X^n) \to D^b(X^{[n]}).
\end{equation}
\end{thm}

An important part of this result is that the equivalences $\Phi$ and $\Psi$ are explicitly computable. In particular Scala \cite{scala} was able to compute the image under $\Phi$ of the tautological bundles $E^{[n]}$. More precisely, consider the space $X\times X^n = \{(P_0,\dots,P_n)\}$ with the two projections
\begin{equation}\label{Xn+1} 
pr_0 \colon X\times X^n \to X, \quad (x_0,\dots,x_n)\mapsto x_0,  \qquad pr_{[1,n]}\colon X\times X^n \to X^n, \quad (x_0,\dots,x_n) \mapsto (x_1,\dots,x_n) 
\end{equation}
and the subscheme
\begin{equation}\label{D}
D_n \subseteq X\times X^n, \qquad D_n = \Delta_{01}\cup \Delta_{02} \cup \dots \cup \Delta_{0n} 
\end{equation}
where $\Delta_{ij}$ denotes the partial diagonal $\Delta_{ij}=\{(x_0,\dots,x_n) \,|\, x_i=x_j\}$. Scala showed the following in \cite[Theorem 2.2.2]{scala}:

\begin{thm}[Scala]\label{scala}
  Let $E$ be a vector bundle on $X$ and let $E^{[n]}$ be the corresponding tautological bundle on $X^{[n]}$. Then $\Phi(E^{[n]})\cong Rpr_{[1,n],*} (pr_0^* E \otimes \mathcal{O}_{D_n})$. Moreover, $\Phi(E^{[n]})$ is concentrated in degree zero, and there is a quasi-isomorphism in $D^b_{\mathfrak{S}_n}(X^n)$
  \begin{equation}
    \Phi(E^{[n]}) \cong 0 \to \mathsf{C}^0_E \to \mathsf{C}^1_E \to \dots \to \mathsf{C}^n_E \to 0
  \end{equation}
for a certain explicit complex $\mathsf{C}^{\bullet}_{E}$.
\end{thm}

\begin{rmk}\label{rmkCi}
In particular, the first term of the complex $\mathsf{C}^{\bullet}_E$ is
\begin{equation}\label{C0}
 \mathsf{C}^0_E :\df \bigoplus_{i=1}^n pr_i^* E.
\end{equation}
For the other terms, we are not going to give an explicit description, since we will not use it later. However we will need the following key property proven by Krug in \cite[Proof of Lemma 3.3]{krug}.
\end{rmk}

\begin{thm}[Krug]\label{krugextensions}
Let $E$ be a vector bundle on $X$. Then for all $i\geq 0$ we have  
\begin{equation}
 \mathcal{E}xt^j_{X^n}(\mathsf{C}^i_E,\mathcal{O}_{X^n}) = 0, \qquad \text{ for } j\ne 2i.
\end{equation}
\end{thm}

\subsection{Higher order embeddings via Hilbert schemes}

We can phrase the concept of $p$-very ampleness from the Introduction in terms of tautological bundles. Let $B$ be a line bundle on $X$ and consider the evaluation map $H^0(X,B)\otimes \mathcal{O}_X\to B$. Pulling back the map to $\Xi^{[p+1]}$ and pushing forward to $X^{[p+1]}$, we obtain another evaluation map
\begin{equation}\label{evmap}
\operatorname{ev}_{B}\colon H^0(X,B)\otimes_{\mathbb{C}} \mathcal{O}_{X^{[p+1]}} \to B^{[p+1]}.
\end{equation}
It can be seen that the fiber of the map over each point $\xi\in X^{[p+1]}$ is precisely the map $\operatorname{ev}_{\xi}$ of (\ref{evmapX}), so that $B$ is $p$-very ample on $X$ if and only if the evaluation map (\ref{evmap}) is surjective. Moreover, $B$ is $p$-spanned if and only if the map (\ref{evmap}) is surjective when restricted to the open subset of curvilinear subschemes $U_{p+1}\subseteq X^{[p+1]}$.

There is also  a connection between tautological bundles and jet very ampleness for surfaces, which is stated already in \cite{ein_lazarsfeld_yang} in a different language. Let $B$ be a line bundle on $X$: in \cite[Lemma 1.5]{ein_lazarsfeld_yang}, the authors construct a coherent sheaf $\mathcal{E}_{p+1,B}$ on $X^{p+1}$ such that the fiber over a point $(x_1,\dots,x_{p+1}) \in X^{p+1}$ is given by
\begin{equation}\label{fiberEB} 
{\mathcal{E}_{p+1,B}}_{|(x_1,\dots,x_{p+1})} = H^0(X,B \otimes (\mathcal{O}_X/\mathfrak{m}_\zeta)), \qquad \zeta = x_1 + \dots +x_{p+1}. 
\end{equation}
Moreover, they construct an evaluation map 
\begin{equation}\label{evEB}  
H^0(X,B) \otimes \mathcal{O}_{X^{p+1}} \to \mathcal{E}_{p+1,B} 
\end{equation}
which on fibers coincides with (\ref{jetcondition}), so that $B$ is $p$-jet very ample if and only if this map of sheaves is surjective. Looking at the construction of \cite{ein_lazarsfeld_yang}, one actually sees that $\mathcal{E}_{p+1,B}$ is obtained as $\mathcal{E}_{p+1,B} \cong pr_{[1,p+1],*}(pr_0^*B \otimes \mathcal{O}_{D_{p+1}})$, where we are using the notation of (\ref{Xn+1}) and (\ref{D}). And this is precisely what appears in Scala's Theorem \ref{scala}, that we can then rephrase as follows.

\begin{cor}\label{scala2}
Let $B$ be a line bundle on $X$ and $p\geq 0$ an integer. Then $\Phi(B^{[p+1]}) \cong \mathcal{E}_{p+1,B}$ in $D^b_{\mathfrak{S}_n}(X^{p+1})$ and the evaluation map (\ref{evEB}) corresponds to the map
\begin{equation} 
H^0(X,B) \otimes_{\mathbb{C}} \Phi(\mathcal{O}_{X^{[p+1]}}) \to \Phi(B^{[p+1]}) 
\end{equation}
that we obtain applying the functor $\Phi$ to the evaluation map (\ref{evmap}).
\end{cor}

\section{Hilbert schemes and asymptotic syzygies}

The fundamental connection between Hilbert schemes and syzygies was estabilished by Voisin in \cite{voisin_even}. Again, let $X$ be a smooth projective surface, $L$ an ample and globally generated line bundle on $X$ and $B$ another line bundle. We also fix an integer $p\geq 0$. Then we have the evaluation map
\begin{equation}\label{evmap2} 
\operatorname{ev}_B\colon H^0(X,B) \otimes \mathcal{O}_{X^{[p+1]}} \to B^{[p+1]} 
\end{equation}
that we can twist by $\operatorname{det} L^{[p+1]}$ to get another map
\begin{equation}\label{mapbundles}
\operatorname{ev}_{B,L}\colon H^0(X,B)\otimes \det L^{[p+1]} \to B^{[p+1]}\otimes \det L^{[p+1]}. 
\end{equation}

Building on work of Voisin, Ein and Lazarsfeld realized that one can compute the Koszul cohomology groups from the map induced by $\operatorname{ev}_{B,L}$ on global sections. They proved it in \cite[Lemma 1.1]{ein_lazarsfeld} for smooth curves and we show it here in the case of surfaces.

\begin{lemma}[Voisin, Ein-Lazarsfeld]\label{einlazarsfeldvoisin}
  Let $L$ be an ample and globally generated line bundle on $X$ and $B$ any line bundle. Then
  \begin{equation}\label{mapglobalsections}
    K_{p,1}(X,B,L)= \operatorname{coker} \left[ 
H^0(X,B)\otimes H^0(X^{[p+1]}, \det L^{[p+1]}) \to H^0(X^{[p+1]},B^{[p+1]}\otimes \det L^{[p+1]}) \right]. 
\end{equation}
In particular $K_{p,1}(X,B,L)=0$ if and only if  the map $\operatorname{ev}_{B,L}$ (\ref{mapbundles}) is surjective on global sections.
\end{lemma}
\begin{proof}
Let $U=U_{p+1}\subseteq X^{[p+1]}$ be the open subset of curvilinear subschemes. Denote also by $\Xi^{[p+1]}_U$ the corresponding universal family: more precisely $\Xi^{[p+1]}_U = \Xi^{[p+1]} \cap (X \times U)$. We also denote by $L^{[p+1]}_U$ the restriction of $L^{[p+1]}$ to $U$. By definition we have $\Xi^{[p+1]}_U \subseteq X \times U$ so that we can consider the restriction map
\begin{equation}\label{restriction}
H^0(X \times U, B \boxtimes \det L^{[p+1]}_{U}) \to H^0(\Xi^{[p+1]}_{U}, ( B \boxtimes \det L^{[p+1]}_{U})_{|\Xi^{[p+1]}_{U}} ).
\end{equation}
Voisin proved that $K_{p,1}(X,B,L)$ coincides with the cokernel of this map \cite[Corollary 5.5, Remark 5.6]{aprodu_nagel}. Now we want to rewrite (\ref{restriction}). By definition, we see that it is the map induced on global sections by the morphism of sheaves on $X\times U$:
\begin{equation}\label{restrictionsheaves}
pr_X^*B \otimes pr_U^*(\det L^{[p+1]}_U) \to (pr_X^*B \otimes pr_U^*(\det L^{[p+1]}_U))\otimes \mathcal{O}_{\Xi^{[p+1]}_U} .
\end{equation}
Hence, we can look at (\ref{restriction}) also as the map induced on global sections by the pushforward of (\ref{restrictionsheaves}) along $pr_U$. By the projection formula we can write this pushforward as
\begin{equation}
pr_{U,*}(pr_X^*B)\otimes \det L^{[p+1]}_U  \to  pr_{U,*}(pr_X^*B \otimes \mathcal{O}_{\Xi^{[p+1]}_U}) \otimes \det L^{[p+1]}_U .
\end{equation}
Now, using the definition of tautological bundles, together with flat base change along $U\hookrightarrow X^{[p+1]}$  we can rewrite this as
\begin{equation}\label{sheavesU}
H^0(X,B) \otimes_{\mathbb{C}} \det L^{[p+1]}_U  \to B^{[p+1]}_{U} \otimes \det L^{[p+1]}_U
\end{equation}
where the map is actually the restriction of the evaluation map $\operatorname{ev}_{B,L}$ (\ref{mapbundles}) to $U$. Using the fact that $X^{[p+1]}$ is normal and that the complement of $U$ has codimension at least two (see Remark \ref{sizecurvilinear}), we see that the map induced by (\ref{sheavesU}) on global sections is the same as the map (\ref{mapglobalsections}) and we conclude. 
\end{proof}

\begin{rmk}
Since $K_{p,q}(X,B,L) = K_{p,1}(X,B\otimes L^{\otimes (q-1)},L)$, the previous lemma gives a representation of every Koszul cohomology group.
\end{rmk}

Using this lemma, we want to study the vanishing of $K_{p,1}(X,B,L)$ when $L\gg 0$. The idea is to pushforward the map $\operatorname{ev}_{B,L}$ (\ref{mapbundles}) to the symmetric product via the Hilbert-Chow morphism $\mu\colon X^{[p+1]} \to X^{(p+1)}$. This allows us to give a characterization of the vanishing of $K_{p,1}(X,B,L)$ purely in terms of $B$.

We first need an easy lemma. We give the proof for completeness.

\begin{lemma}\label{easylemma}
Let $X$ be a projective scheme and $\phi\colon \mathcal{F}\to \mathcal{G}$ a map of coherent sheaves on $X$. Then $\phi$ is surjective if and only if the induced map $\mathcal{F} \otimes L \to \mathcal{G}\otimes L$ is surjective on global sections when $L\gg 0$.
\end{lemma}
\begin{proof}
  We have an exact sequence of sheaves
  \begin{equation}
    0 \to \operatorname{Ker } \phi \to \mathcal{F} \to \mathcal{G} \to \operatorname{Coker } \phi \to 0
  \end{equation}
  and for $L\gg 0$ we have that $H^1(X,\operatorname{Ker } \phi \otimes L)=H^1(X,\operatorname{Im } \phi \otimes L)=0$ thanks to Serre vanishing. Hence, on global sections we obtain an exact sequence
  \begin{equation}
    0 \to H^0(X,\operatorname{Ker } \phi \otimes L) \to H^0(X,\mathcal{F}\otimes L) \to H^0(X,\mathcal{G}\otimes L) \to H^0(X,\operatorname{Coker } \phi \otimes L) \to 0.
  \end{equation}
Since $L\gg 0$, the sheaf $\operatorname{Coker } \phi\otimes L$ is globally generated, so that $H^0(X,\operatorname{Coker } \phi \otimes L)=0$ if and only if $\operatorname{Coker } \phi =0$. But this is exactly what we want to prove.
\end{proof}

Now we can state our criterion. In what follows, we will denote by $\mathfrak{a}_n$ the alternating representation of $\mathfrak{S}_n$: then from any $\mathfrak{S}_n$-equivariant sheaf $\mathcal{F}$ on $X^n$, we can get another one by $\mathcal{F}\otimes \mathfrak{a}_n$, and the same holds for complexes in the derived category $D^b_{\mathfrak{S}_n}(X^n)$. It is easy to see that tensoring by $\mathfrak{a}_n$ is an exact functor.

\begin{prop}\label{criterion}
  Let $X$ be a smooth projective surface and $B$ a line bundle on $X$. Then $K_{p,1}(X,B,L)=0$ for $L\gg 0$ if and only if the induced map of sheaves on $X^{(p+1)}$
  \begin{equation}\label{characterizationfirstmap}
    H^0(X,B)\otimes_{\mathbb{C}} \mu_*(\mathcal{O}(-\delta_{p+1})) \to \mu_*(B^{[p+1]}\otimes \mathcal{O}(-\delta_{p+1}))
  \end{equation}
  is surjective. Moreover, this map is isomorphic to the map
  \begin{equation}\label{characterizationsecondmap}
   H^0(X,B)\otimes  \pi_*^{\mathfrak{S}_{p+1}}(\mathcal{O}_{X^n} \otimes \mathfrak{a}_{p+1}) \to \pi^{\mathfrak{S}_{p+1}}_* (\mathcal{E}_{p+1,B} \otimes \mathfrak{a}_{p+1}).
  \end{equation}
\end{prop}
\begin{proof}
  We know from Lemma  \ref{einlazarsfeldvoisin} that $K_{p,1}(X,B,L)=0$ if and only if the map
  \begin{equation} H^0(X,B)\otimes_{\mathbb{C}} \det L^{[p+1]} \to B^{[p+1]}\otimes \det L^{[p+1]} \end{equation}
  is surjective on global sections. Taking the pushforward along $\mu$, this is equivalent to saying that
  \begin{equation}  H^0(X,B)\otimes_{\mathbb{C}} \mu_*(\det L^{[p+1]}) \to \mu_*(B^{[p+1]}\otimes \det L^{[p+1]}) \end{equation}
  is surjective on global sections. However, since $\det L^{[p+1]} = \mathcal{O}(-\delta_{p+1}) \otimes \mu^*L^{({p+1})}$ by (\ref{dettaut}), we can rewrite the last map using the projection formula as
  \begin{equation}
    \left(H^0(X,B) \otimes_{\mathbb{C}} \mu_*(\mathcal{O}(-\delta_{p+1}))\right)\otimes L^{(p+1)} \to \mu_*(B^{[p+1]}\otimes \mathcal{O}(-\delta_{p+1})) \otimes L^{(p+1)}.
  \end{equation}
  Now, Lemma \ref{ampleness}  shows that  $L\gg 0$ implies $L^{(p+1)} \gg 0$ as well, and then Lemma \ref{easylemma} shows that this map is surjective on global sections if and only if the map (\ref{characterizationfirstmap}) is surjective.

  To conclude, we need to show that the maps (\ref{characterizationfirstmap}) and (\ref{characterizationsecondmap}) are isomorphic: to do this we will use the equivalences in Haiman's Theorem \ref{mckay}. First, Krug has proven in \cite[Theorem 1.1]{krug_mckay} that $\mathcal{O}(-\delta_{p+1}) \cong \Psi(\mathcal{O}_{X^{p+1}} \otimes \mathfrak{a}_{p+1})$, so that we can rewrite (\ref{characterizationfirstmap}) as 
  \begin{equation}
   H^0(X,B)\otimes \mu_*(\Psi(\mathcal{O}_{X^{p+1}}\otimes \mathfrak{a}_{p+1})) \to \mu_*(B^{[p+1]}\otimes \Psi(\mathcal{O}_{X^{p+1}}\otimes \mathfrak{a}_{p+1})).
 \end{equation}
 Now, using \cite[Proposition 5.1]{krug_mckay} and \cite[Proposition 1.3.3]{scala}, we get functorial isomorphisms in $D^b(X^{(p+1)})$:
 \begin{small} 
 \begin{equation}
 \mu_*(\Psi(\mathcal{O}_{X^{p+1}}\otimes \mathfrak{a}_{p+1})) \cong \pi_*^{\mathfrak{S}_{p+1}}(\mathcal{O}_{X^{p+1}}\otimes \mathfrak{a}_{p+1}), \qquad \mu_*(B^{[p+1]}\otimes \Psi(\mathcal{O}_{X^{p+1}}\otimes \mathfrak{a}_{p+1})) \cong \pi_*^{\mathfrak{S}_{p+1}}(\Phi(B^{[p+1]}) \otimes \mathfrak{a}_{p+1} )
\end{equation}
\end{small} 
so that the map (\ref{characterizationfirstmap}) corresponds to
\begin{equation}
H^0(X,B)\otimes  \pi_*^{\mathfrak{S}_{p+1}}(\mathcal{O}_{X^{p+1}} \otimes \mathfrak{a}_{p+1}) \to \pi^{\mathfrak{S}_{p+1}}_* (\Phi(B^{[p+1]}) \otimes \mathfrak{a}_{p+1})
\end{equation}
and since $\Phi(B^{[p+1]}) \cong \mathcal{E}_{p+1,B}$ by Corollary \ref{scala2}, we conclude.
\end{proof}

To illustrate the criterion of Proposition \ref{criterion} we use it to give alternative proofs to Theorems A and B from \cite{ein_lazarsfeld_yang} in the case of surfaces:

\begin{cor}\cite[Theorem A]{ein_lazarsfeld_yang}\label{thmAreproof}
Let $X$ be a smooth projective surface and $B$ a $p$-jet very ample line bundle on $X$. Then $K_{p,1}(X,B,L)=0$ for $L\gg 0$.
\end{cor}
\begin{proof}
By Proposition \ref{criterion}, $K_{p,1}(X,B,L)=0$ for $L\gg 0$ if and only if the map 
\begin{equation}
   H^0(X,B)\otimes  \pi_*^{\mathfrak{S}_{p+1}}(\mathcal{O}_{X^n} \otimes \mathfrak{a}_{p+1}) \to \pi^{\mathfrak{S}_{p+1}}_* (\mathcal{E}_{p+1,B} \otimes \mathfrak{a}_{p+1}).
  \end{equation}
  is surjective. The assumption that $B$ is $p$-jet very ample means that the map
  \begin{equation}
H^0(X,B)\otimes \mathcal{O}_{X^{p+1}} \to \mathcal{E}_{p+1,B}
\end{equation}
is surjective. Since both functors of tensoring by $\mathfrak{a}_{p+1}$ and taking pushforward $\pi^{\mathfrak{S}_{p+1}}_*$ are exact, it follows that the first map is surjective as well.
\end{proof}

\begin{cor}\cite[Theorem B]{ein_lazarsfeld_yang}\label{thmBreproof}
Let $X$ be a smooth projective surface and $B$ a line bundle on $X$. If $K_{p,1}(X,B,L)=0$ for $L\gg 0$, then  the evaluation map
\begin{equation} \operatorname{ev}_{\xi}\colon  H^0(X,B) \to H^0(X,B\otimes \mathcal{O}_{\xi}) \end{equation}
is surjective for any subscheme $\xi\in X^{[p+1]}$ consisting of $p+1$ distinct points.
\end{cor}
\begin{proof}
By Proposition  \ref{criterion}, if $K_{p,1}(X,B,L)=0$ for $L\gg 0$ then the map
\begin{equation}
  H^0(X,B)\otimes \mu_*\mathcal{O}(-\delta_{p+1}) \to \mu_*(B^{[p+1]}\otimes \mathcal{O}(-\delta_{p+1}))
\end{equation}
is surjective. This map restricted to the open subset $V\subseteq X^{(p+1)}$ consisting of reduced cycles is again surjective. Now it is easy to see that $\mu_{|\mu^{-1}(V)}\colon \mu^{-1}(V) \to V$ is an isomorphism, so that the map
\begin{equation}
  H^0(X,B) \otimes \mathcal{O}(-\delta_{p+1}) \to B^{[p+1]} \otimes \mathcal{O}(-\delta_{p+1})
\end{equation}
is surjective on $\mu^{-1}(V)$. Tensoring by $\mathcal{O}(\delta_{p+1})$ we obtain what we want. 
\end{proof}

\section{Higher order embeddings and asymptotic syzygies on surfaces}

Using Proposition \ref{criterion}, we can prove Theorem B from the Introduction. The key conditions are some local cohomological vanishing for tautological bundles.  We would like to thank Victor Lozovanu for a discussion about the following proposition.
\begin{prop}\label{condition}
	Let $X$ be a smooth projective surface, $p\geq 0$ an integer and suppose that
	\begin{equation}\label{conditionO}
	R^{i+1}\mu_*({\operatorname{Sym}^i \mathcal{O}_X^{[p+1]}}^{\vee})=0  \qquad \text{for all } 0\leq i < p. 
	\end{equation}
	Let also $B$ be a $p$-very ample line bundle on $X$.
	Then $K_{p,1}(X,B,L)=0$ for $L\gg 0$. 
\end{prop}

\begin{proof}
	Using Proposition \ref{criterion}, we need to prove that the map of sheaves on $X^{(p+1)}$ 
	\begin{equation} 
	H^0(X,B)\otimes \mu_*\mathcal{O}(-\delta_{p+1}) \to \mu_*(B^{[p+1]}\otimes \mathcal{O}(-\delta_{p+1}))
	\end{equation}
	is surjective. This map is surjective if and only if it is surjective when tensored by the line bundle $B^{(p+1)}$. Using (\ref{dettaut}) and the projection formula, we can rewrite the tensored map as
	\begin{equation}\label{ausiliarymap}
	\mu_*(H^0(X,B)\otimes_{\mathbb{C}} \det B^{[p+1]}) \to \mu_*(B^{[p+1]}\otimes \det B^{[p+1]}).
	\end{equation}
	Set $h^0(X,B)=r+1$. Taking the Buchsbaum-Rim complex \cite[Theorem B.2.2]{pag1} associated to the surjective map $H^0(X,B)\otimes \mathcal{O}_{X^{^{[p+1]}}} \to B^{[p+1]}$ and tensoring by $\det B^{[p+1]}$ we get an exact complex of vector bundles
	\begin{equation}
	0 \to E_{r-p-1} \to \dots \to E_{1} \to E_{0} \to H^0(X,B) \otimes \det B^{[p+1]} \to B^{[p+1]}\otimes \det B^{[p+1]} \to 0
	\end{equation}
	with
	\begin{equation}
	E_{i} = \wedge^{p+2+i} H^0(X,B) \otimes_{\mathbb{C}} \operatorname{Sym}^{i}(B^{[p+1]})^{\vee}.
	\end{equation}
	Breaking this complex into short exact sequences, we see that if $R^{i+1}\mu_*(E_i)=0$ for all $0\leq i \leq r-p-1$, then the map (\ref{ausiliarymap}) is surjective. However, a result of Brian\c{c}on \cite{briancon} shows that the fibers of the Hilbert-Chow morphism have dimension at most $p$, hence it is enough to have $R^{i+1}\mu_*(E_i)=0$ for all $0\leq i < p$. This is the same as 
	\begin{equation}\label{conditionB}
	R^{i+1}\mu_*({\operatorname{Sym}^i B^{[p+1]}}^{\vee})=0 \text{  for all } 0\leq i <p.
	\end{equation} 
	Now, Scala shows in \cite[Lemma 3.1]{scalasymm} that we can find an open cover of $X^{p+1}$ composed of sets $V^{p+1}$, where $V\subseteq X$ is an open affine subset where $B$ is trivial. Then it follows by the construction of the symmetric product and the Hilbert scheme, that we have an open cover of $X^{(p+1)}$ of the form $V^{(p+1)}$ and that over these sets the Hilbert-Chow morphism restricts to $\mu_V\colon V^{[p+1]} \to V^{(p+1)}$. To conclude, it is straightforward to show that 
	\begin{align*}
	R^{i+1}\mu_*({\operatorname{Sym}^i B^{[p+1]}}^{\vee})_{|V^{(p+1)}} & \cong R^{i+1}{\mu_V}_*(\operatorname{Sym}^i ((B_{|V})^{[p+1]})^{\vee}) \\
	& \cong R^{i+1}{\mu_V}_*(\operatorname{Sym}^i ((\mathcal{O}_V^{[p+1]})^{\vee}) \cong R^{i+1}\mu_*({\operatorname{Sym}^i \mathcal{O}_X^{[p+1]}}^{\vee})_{|V^{(p+1)}}.
	\end{align*}
	In particular, condition (\ref{conditionB}) is equivalent to hypothesis (\ref{conditionO}).
\end{proof}

To conclude the proof of Theorem B we need to verify the cohomological vanishings of Proposition \ref{condition}. We first spell out some consequences of Grothendieck duality for the morphisms $\mu\colon X^{[n]}\to X^{(n)}$ and $\pi\colon X^n \to X^{(n)}$.

\begin{lemma}\label{duality}
  Let $X$ be a smooth projective surface. For all $F\in D^b(X^{[n]})$ and $G\in D^b_{\mathfrak{S}_n}(X^n)$ we have the isomorphisms in $D^b(X^{(n)})$:
\begin{align}
    R\mu_*(R\mathcal{H}om_{X^{[n]}}(F,\mathcal{O}_{X^{[n]}})) &\cong R\mathcal{H}om_{X^{(n)}}(R\mu_*(F),\mathcal{O}_{X^{(n)}}),\\
    \pi_*^{\mathfrak{S}_n}(R\mathcal{H}om_{X^n}(G,\mathcal{O}_{X^n})) &\cong R\mathcal{H}om_{X^{(n)}}(\pi_*^{\mathfrak{S}_n}(G),\mathcal{O}_{X^{(n)}}).
\end{align}
\end{lemma}
\begin{proof}
The first statement follows from the usual Grothendieck duality applied to the morphism $\mu\colon X^{[n]} \to X^{(n)}$, together with the fact that both $X^{[n]},X^{(n)}$ are Gorenstein and $\mu^*\omega_{X^{(n)}} \cong \omega_{X^{[n]}}$. The other one follows from equivariant Grothendieck duality, see for example \cite[Theorem 1.0.2]{abuaf}.
\end{proof}

Now we prove the vanishings:

\begin{lemma}\label{vanishings}
  Let $X$ be a smooth surface.  Then
  \begin{equation}
   R^{1}\mu_*(\mathcal{O}_{X^{[n]}}) = 0, \qquad   R^2\mu_*((\mathcal{O}_{X}^{[n]})^{\vee})=0, \qquad  R^3\mu_*((\mathcal{O}_{X}^{[n]}\otimes \mathcal{O}_{X}^{[n]})^{\vee}) = 0.
  \end{equation}
\end{lemma}
\begin{proof}
  For the first vanishing, we know that $R\mu_*\mathcal{O}_{X^{[n]}}=\mathcal{O}_{X^{(n)}}$ because $X^{(n)}$ has rational singularities and $\mu\colon X^{[n]} \to X^{(n)}$ is a resolution. For the second, we see from \cite[Theorem 3.2.1]{scala} that
  \begin{equation}
   R\mu_* \mathcal{O}_{X}^{[n]} \cong \pi_*^{\mathfrak{S}_n}(\mathsf{C}^0_{\mathcal{O}_X})
 \end{equation}
 and then Lemma \ref{duality} shows that
 \begin{equation}
R\mu_*((\mathcal{O}_{X}^{[n]})^{\vee}) \cong \pi_*^{\mathfrak{S}_n}((\mathsf{C}^0_{\mathcal{O}_X})^{\vee}).
\end{equation}
Since $\mathsf{C}^0_{\mathcal{O}_X}$ is locally free and $\pi_*^{\mathfrak{S}_n}$ an exact functor, it follows that $R\mu_*((\mathcal{O}_{X}^{[n]})^{\vee})$ is concentrated in degree zero, and in particular $R^2\mu_*((\mathcal{O}_{X}^{[n]})^{\vee})=0$.

For the third vanishing we observe that $R\mu_*(\mathcal{O}_X^{[n]}\otimes \mathcal{O}_X^{[n]})$ is concentrated in degree zero \cite[Corollary 3.3.1]{scala}. Hence, using the first part of Lemma \ref{duality}, we get that 
  \begin{equation}
    R^3\mu_*((\mathcal{O}_{X}^{[n]}\otimes \mathcal{O}_{X}^{[n]})^{\vee}) \cong \mathcal{E}xt^{3}_{X^{(n)}}(\mu_*(\mathcal{O}_{X}^{[n]} \otimes \mathcal{O}_X^{[n]}),\mathcal{O}_{X^{(n)}}).
  \end{equation}
  Now, Scala gives in \cite[Theorem 3.5.2]{scala} an exact sequence of sheaves on $X^{(n)}$:
  \begin{equation}
    0 \to \mu_*(\mathcal{O}_{X}^{[n]} \otimes \mathcal{O}_X^{[n]}) \to \pi_*^{\mathfrak{S}_n}(\mathsf{C}^0_{\mathcal{O}}\otimes \mathsf{C}^0_{\mathcal{O}}) \to \pi_*^{\mathfrak{S}_n}(\mathsf{C}^1_{\mathcal{O}}\otimes \mathsf{C}^0_{\mathcal{O}}) \to 0
  \end{equation}
  where the $\mathsf{C}^i_{\bullet}$ are the sheaves appearing in Theorem \ref{scala}.
  Therefore, it is enough to show
  \begin{equation}
   \mathcal{E}xt^{3}_{X^{(n)}}(\pi_*^{\mathfrak{S}_n}(\mathsf{C}^0_{\mathcal{O}}\otimes \mathsf{C}^0_{\mathcal{O}}),\mathcal{O}_{X^{(n)}}) = 0 \qquad \mathcal{E}xt^{4}_{X^{[n]}}(\pi_*^{\mathfrak{S}_n}(\mathsf{C}^1_{\mathcal{O}}\otimes \mathsf{C}^0_{\mathcal{O}}),\mathcal{O}_{X^{(n)}}) = 0.
 \end{equation}
 For the first one we see, using Lemma \ref{duality}, that
 \begin{equation}
\mathcal{E}xt^{3}_{X^{(n)}}(\pi_*^{\mathfrak{S}_n}(\mathsf{C}^0_{\mathcal{O}}\otimes \mathsf{C}^0_{\mathcal{O}}),\mathcal{O}_{X^{(n)}}) \cong \pi^{\mathfrak{S}_n}_*(\mathcal{E}xt_{X^n}(\mathsf{C}^0_{\mathcal{O}}\otimes \mathsf{C}^0_{\mathcal{O}},\mathcal{O}_{X^n})) = 0
\end{equation}
where the last vanishing follows from the fact that $\mathsf{C}^0_{\mathcal{O}}$ is locally free. For the second, we use again Lemma \ref{duality} and we get
\begin{equation}
\mathcal{E}xt^{4}_{X^{[n]}}(\pi_*^{\mathfrak{S}_n}(\mathsf{C}^1_{\mathcal{O}}\otimes \mathsf{C}^0_{\mathcal{O}}),\mathcal{O}_{X^{(n)}}) \cong \pi_*^{\mathfrak{S}_n}\mathcal{E}xt^4_{X^n}(\mathsf{C}^1_{\mathcal{O}}\otimes \mathsf{C}^0_{\mathcal{O}},\mathcal{O}_{X^n}) \cong \pi_*^{\mathfrak{S}_n}\left( {\mathsf{C}^0_{\mathcal{O}}}^{\vee}\otimes \mathcal{E}xt^4_{X^n}(\mathsf{C}^1_{\mathcal{O}},\mathcal{O}_{X^n})\right),
\end{equation}
where in the last step we have used again the fact that $\mathsf{C}^0_{\mathcal{O}}$ is locally free. To conclude we observe that  
\begin{equation}
\mathcal{E}xt^4_{X^n}(\mathsf{C}^1_{\mathcal{O}},\mathcal{O}_{X^n})=0
\end{equation}
by  Theorem \ref{krugextensions}.
\end{proof}

Now it is straightforward to prove Theorem B:

\begin{proof}[Proof of Theorem B]
Let $X$ be a smooth projective surface and $B$ a line bundle on $X$. Fix also an integer $0\leq p\leq 3$. If $K_{p,1}(X,B,L)=0$ for $L\gg 0$, Theorem A gives that $B$ is $p$-very ample. We prove the converse through Proposition \ref{condition}. We need to check the vanishings in (\ref{conditionO}):the cases $p=0,1,2$ follow immediately from Lemma \ref{vanishings}. For the case $p=3$, we use again Lemma \ref{vanishings}, together with the observation that $\operatorname{Sym}^2((\mathcal{O}_X^{[n]})^{\vee})$ is a direct summand of $(\mathcal{O}_{X}^{[n]}\otimes \mathcal{O}_{X}^{[n]})^{\vee}$. 
\end{proof}

\newpage
\subsection{Concluding remarks}
\begin{itemize}
\item It is possible that the statement of Theorem B remains true for any $p$, but the key part in our proof was to check the vanishings
  \begin{equation}\label{conjecturalvanishing}
R^{k+1}\mu_*\left( \mathcal{H}om_{X^{[n]}}\left(\left(\mathcal{O}_{X}^{[n]}\right)^{\otimes k},\mathcal{O}_{X^{[n]}}\right) \right) = 0 
\end{equation}
and in our case we were able to do so because Scala \cite{scala} gives a relatively simple description of the sheaves $R\mu_*((\mathcal{O}_{X}^{[n]})^{\otimes k})$, when $k=0,1,2$ is small. However, as $k$ increases, these sheaves become increasingly more complicated, and it is not clear whether it is possible to check the vanishings explicitly as we have done in Lemma \ref{vanishings}.

Here we would like to discuss informally another point of view on the problem, and argue that the above statement is essentially combinatorial.
We first observe that in the proof  of Lemma \ref{vanishings} we did not use anything about the particular geometry of $X$. Indeed, reasoning as in \cite[p. 8]{scala}, we can look at the vanishings (\ref{conjecturalvanishing}) as being basically local statements on $X$, so that we can restrict to the case of $X=\mathbb{A}^2_{\mathbb{C}}$.
Moreover, using Lemma \ref{duality} and Proposition \cite[1.7.2 (a)]{scala}, we see that
\begin{equation}
R^{k+1}\mu_*\left( \mathcal{H}om_{X^{[n]}}\left(\left(\mathcal{O}_{X}^{[n]}\right)^{\otimes k},\mathcal{O}_{X^{[n]}}\right) \right) \cong \pi_*^{\mathfrak{S}_n} \left(\mathcal{E}xt^{k+1}_{X^n}\left( \Phi\left(\left(\mathcal{O}_{X}^{[n]}\right)^{\otimes k}\right),\mathcal{O}_{X^n} \right) \right) 
\end{equation}
In the case of $X=\mathbb{A}^2_{\mathbb{C}}$, Haiman gave an explicit description of $\Phi((\mathcal{O}_{X}^{[n]})^{\otimes k})$ for any  $k\geq 0$. To state his result, let $S=\mathbb{C}[X_1,Y_1,\dots,X_n,Y_n]$ be the ring of $X^n$, with the natural action of $\mathfrak{S}_n$, and let $S[A_1,B_1,\dots,A_k,B_k]$ be the ring of $X^n\times X^k$.  For every function $f\colon \{1,\dots, k \} \to \{1, \dots, n \}$, define a linear subspace $W_f\subseteq X^n \times X^k$ by
\begin{equation}
W_f = \operatorname{Spec} S[A_1,B_1,\dots,A_k,B_k]/I_f, \qquad I_f = \left( A_i-X_{f(i)}, B_i - Y_{f(i)} \,\mid \, i=1,\dots,k  \right)
\end{equation}
and set $Z(n,k)\subseteq X^n\times X^k$ to be the union of these subspaces. The scheme $Z(n,k)$ is called the \textit{Haiman polygraph} and its coordinate ring is by definition
\begin{equation}
Z(n,k) = \operatorname{Spec} R(n,k), \qquad R(n,k) =  S[A_1,B_1,\dots,A_k,B_k] / \bigcap_{f} I_f.
\end{equation}
In \cite[Theorem 2.1, Proposition 5.3]{haiman_2} Haiman proved the following:
\begin{equation}
\Phi\left(\left(\mathcal{O}_{X}^{[n]}\right)^{\otimes k}\right) \cong R(n,k)
\end{equation}
where we look at $R(n,k)$ as a $\mathfrak{S}_n$-linearized  coherent sheaf on $X^n$. In particular, the vanishings in (\ref{conjecturalvanishing}) correspond to
\begin{equation}
\operatorname{Ext}^{k+1}_{S}\left( R(n,k), S \right)^{\mathfrak{S}_n} = 0
\end{equation}
so that we can regard Theorem B as a consequence of this essentially combinatorial statement about the ring $\mathbb{C}[X,Y]$.
Moreover, this statement is completely explicit and in principle it can be verified by a computer. We wrote a program to check these vanishings, but the problem becomes computationally very expensive as $n$ and $k$ grow, and we were not able to obtain better results than those already proved before.  
\vspace{5pt} 

\item A topic that we do not discuss at all is how to make the statement of Theorem B effective. Indeed, for a curve $C$,  Ein and Lazarsfeld give in in \cite[Proposition 2.1]{ein_lazarsfeld} a lower bound on the degree of a line bundle $L$ such that if $B$ is $p$-very ample then $K_{p,1}(C,B,L)=0$. The bound has later been improved by Rathmann \cite{rathmann} for any curve and by Farkas and Kemeny for a general curve and $B=\omega_C$ \cite{farkas_kemeny}. It is then natural to ask for a similar result for surfaces.  
\vspace{5pt}

\item Instead, it is not clear whether one should expect that Theorem B extends to varieties of dimension greater than two. Indeed, the Hilbert scheme of points and its relation with Koszul cohomology  become more complicated in higher dimensions, as observed by Ein, Lazarsfeld and Yang in \cite[Footnote 9]{ein_lazarsfeld_survey}. 
\end{itemize}

\end{document}